\documentclass[11pt, a4paper]{amsart}
\pdfoutput=1

\usepackage[alphabetic]{amsrefs}
\usepackage{amsmath, amssymb, amsthm, bbm, multicol, rotating, enumitem, tabularx, epsfig}
\usepackage[all]{xy}
\usepackage[english]{babel}
\usepackage{textcomp}

\usepackage{graphicx, color}
\usepackage[T1]{fontenc}

\theoremstyle{plain}
\newtheorem{thm}{Theorem}[section]
\newtheorem{prop}[thm]{Proposition}
\newtheorem{lemma}[thm]{Lemma}
\newtheorem{sublem}[thm]{Sublemma}
\newtheorem{cor}[thm]{Corollary}
\newtheorem{quest}[thm]{Question}
\newtheorem*{claim}{Claim}

\newtheorem{metathm}{Theorem}

\theoremstyle{definition}
\newtheorem{defin}[thm]{Definition}
\newtheorem{constr}[thm]{Construction}
\newtheorem{rem}[thm]{Remark}

\newcommand{\E}{\mathbb{E}}
\newcommand{\Z}{\mathbb{Z}}
\newcommand{\R}{\mathbb{R}}
\newcommand{\Hyp}{\mathbb{H}}
\newcommand{\sys}{\hat}
\newcommand{\lk}{\mathrm{lk}}
\newcommand{\st}{\mathrm{st}}

\begin{document}

\title{Systolizing buildings}
\author{Piotr Przytycki and Petra Schwer}
\thanks{Partially supported by the Foundation for Polish Science. The first author was also partially supported by the National Science Centre DEC-2012/06/A/ST1/00259 and the second author by the DFG grant SCHW 1550-01.}

\begin{abstract}
We introduce a construction turning some Coxeter and Davis realizations of buildings into systolic complexes. Consequently groups acting geometrically on buildings of triangle types distinct from $(2,4,4)$, $(2,4,5)$, $(2,5,5)$, and various rank~$4$ types are systolic.
\end{abstract}
\maketitle

\section{Introduction}

A flag simplicial complex is \emph{systolic} if it is simply-connected and all of its vertex links are \emph{$6$--large}, that is all cycles of length $4$ or $5$ have diagonals.
A group is \emph{systolic} if it acts geometrically (i.e.\ properly and cocompactly by automorphisms) on a systolic complex. Systolic complexes and groups were introduced by Januszkiewicz and \'Swi\c{a}tkowski \cite{JS}, and independently by Haglund \cite{H} although their $1$--skeleta were studied earlier by Chepoi and others under the name of \emph{bridged graphs} (see~\cite{BC}). They constructed examples in high dimensions, established analogies with CAT(0) spaces, and studied their exotic filling properties.

In this article we systolize the Coxeter and Davis realizations of buildings of given Coxeter types. We say that a systolic complex $\sys X$ is a \emph{systolization} of its subcomplex $X$, if the action of the group of type preserving automorphisms of $X$ extends to $\sys X$, and $X\subset \sys X$ is a quasi-isometry.

We assume that all buildings have finite thickness. By the \emph{Coxeter realization} of a building we mean the usual realization of a building as a simplicial complex where an apartment is realized as the Coxeter complex of the appropriate Coxeter group. By the \emph{Davis realization} we mean the
subcomplex of the barycentric subdivision of the Coxeter realization obtained by removing
the open stars of vertices of infinite type. Note that this is the barycentric subdivision of what is usually called the Davis complex.

\begin{metathm}\label{thm:dim2}
Let $W$ be a triangle Coxeter group with finite exponents distinct from $(2,4,4), (2,4,5)$ and $(2,5,5)$. Then the Coxeter realization and the Davis realization of a building of type $W$ admit a systolization. Consequently the group $W$ and any group acting geometrically by type preserving automorphisms on a building of type $W$ is systolic.
\end{metathm}

\begin{metathm}
\label{thm:244}
The Coxeter group of type $(2,4,4)$ is not systolic.
\end{metathm}

We believe that Coxeter groups of the other two excluded types are also not systolic.

\begin{metathm}\label{thm:dim3}
Let $W$ be a Coxeter group of rank $4$ with finite exponents. Assume that all of its special rank $3$ subgroups are infinite and not of type $(2,4,4), (2,4,5)$ or $(2,5,5)$. Moreover, assume that there is at most one exponent $2$. Then the Coxeter realization and the Davis realization of a building of type $W$ admit a systolization. Consequently the group $W$ and any group acting geometrically by type preserving automorphisms on the Davis realization of a building of type $W$ is systolic.
\end{metathm}

The groups in Theorem~\ref{thm:dim3} have cohomological dimension $2$. However, we also discuss a possible application pointed out by Januszkiewicz that should give rise to new systolic groups of cohomological dimension $3$.

\medskip

\noindent \textbf{Organization.} In Section~\ref{Sec_prelim} we list basic lemmas allowing to identify systolic complexes. In Section~\ref{sec:244} we prove Theorem~\ref{thm:244}, that the Coxeter group of type $(2,4,4)$ is not systolic. Section~\ref{Sec:236} illustrates our systolizing method for the Coxeter complex of the triangle group $(2,3,6)$. We prove Theorem~\ref{thm:dim2}, up to discussing the Davis realization, in Section~\ref{Sec:dim2}.
In Section~\ref{Sec_dim3} we describe the systolization of the Coxeter realization in Theorem~\ref{thm:dim3}, prove that it is simply-connected and has simply-connected vertex links. To prove that it is systolic it remains to verify $6$--largeness of the edge links, which we do in Section~\ref{sec:edge}. We systolize the Davis realization in Section~\ref{sec:Davis}. Finally, in Section~\ref{sec:TJan} we describe Januszkiewicz's construction.

In the appendix (Section~\ref{sec:app}) we prove that a particular piecewise spherical metric on infinite rank $3$ buildings is CAT(1), and we give a criterion for $\pi$--convexity. We make use of this metric in Section~\ref{Sec_dim3}.

\medskip

\noindent \textbf{Discussion of assumptions.} The groups acting geometrically on spherical buildings are finite, hence systolic, since they act geometrically on a point. Assume then that $W$ is infinite. Infinite buildings of rank $2$ are trees, hence systolic, so the smallest interesting rank is $3$.

Let $W$ be a Coxeter group of rank $3$. If one of the exponents of $W$ is infinite, then
the Davis realizations of buildings of type $W$ retract to trees, and that is why we focus on the case where all the exponents are finite. In that case the Davis realization coincides with the barycentric subdivision of the Coxeter realization.

Let now $W$ be a Coxeter group of rank $4$. We exclude special subgroups of type $(2,4,4), (2,4,5)$ and $(2,5,5)$, since a finitely presented subgroup of a systolic group is systolic \cites{Za,HMP}. Consider first the case where there is an infinite exponent.
If the defining graph of $W$ is a cycle of length~$4$, then the triangles of the Davis realization of type $W$ building
are arranged into squares forming a $\mathcal{VH}$--complex. Groups acting geometrically on $\mathcal{VH}$--complexes
were shown to be systolic in \cite{EP}.

If there is an infinite exponent and the defining graph of $W$ is a not cycle of length $4$, then $W=W_1\ast_{W_3} W_2$, where $W_1,W_2$ are triangle groups and $W_3$ is finite. The Davis realization of a building of type $W$ is equivariantly homotopy equivalent to a tree of Davis realizations of buildings of type $W_1,W_2$. Systolizing these gives a tree of systolic complexes, which is systolic. Hence a group acting geometrically on a Davis realization of a building of type $W$ is systolic. That is why in Theorem~\ref{thm:dim3} we focus on the case where all the exponents are finite.

If all of the special rank $3$ subgroups of $W$ are finite, then $W$ acts geometrically on $\R^3$ or $\Hyp^3$, thus it is not systolic \cite{JS2}.
Since the cases where some but not all of the special rank $3$ subgroups are finite are difficult to handle,
we assume that all of them are infinite.

The complication in the problem comes from the exponents $2$, and that is why we study the simplest case, that is the case where there is at most one exponent $2$. If all the exponents are $\geq 3$, then the Coxeter realization is systolic to begin with.
To systolize the Davis realization consider the face complex of the Coxeter realization and remove open stars of original vertices (see Section~\ref{sec:Davis} for details).

\medskip

\noindent \textbf{Acknowledgements.} We thank Stefan Witzel and the referee for many useful suggestions.

\section{Simplicial nonpositive curvature}
\label{Sec_prelim}
In this section we recall basic definitions and notation used to study systolic complexes. We also give criteria for a complex to be systolic.

\begin{defin}
A simplicial complex $X$ is \emph{flag} if every clique (a set of vertices pairwise connected by edges) spans a simplex. A subcomplex $Y$ of $X$ is \emph{full} if any simplex of $X$ spanned by vertices in $Y$ is in $Y$. A \emph{cycle} in $X$ is a subcomplex of $X$ that is a subdivision of the circle.
A flag simplicial complex is \emph{$k$--large}, for $4\leq  k \leq \infty $, if it has no full cycle of length $<k$.
\end{defin}

In other words a flag simplicial complex $X$ is $k$--large if every cycle of length $4\leq l<k$ has a \emph{diagonal}, that is an edge in $X$ between a pair of non-consecutive vertices of the cycle. Consequently, every closed edge-path $\gamma$ in $X$ of length $<k$ bounds a disc diagram with no interior vertices. In particular, if $\gamma$ is locally embedded, then it has three consecutive vertices $u,v,w$ such that $uw$ is an edge.

\begin{defin}
A simplicial complex is \emph{systolic} if it is connected,
simply-connected, and all its vertex links are $6$--large.
A group is \emph{systolic} if it acts geometrically on a systolic complex.
\end{defin}

While in \cite{JS} authors require in the definition of a systolic complex that the links of all the simplices are $6$--large, this is trivially equivalent with our definition. Note that by \cite[Prop 1.4]{JS} a systolic complex is itself $6$--large. Moreover, by \cite[Thm 4.1(1)]{JS} a systolic complex is contractible.

\begin{lemma}
\label{le:amalgam}
Suppose that a simplicial complex $X$ is obtained from two flag simplicial complexes $A$ and $B$ by gluing them along a simplex.
Then $X$ is $k$--large if and only if $A$ and $B$ are $k$--large.
\end{lemma}

\begin{proof}
This follows from the fact that every cycle in $X$ that is not contained in $A$ or $B$ must have two non-consecutive vertices in $A\cap B$
\end{proof}

\begin{cor}\label{cor:join}
The join of a simplex with a discrete set is $\infty$--large.
\end{cor}

Here is another criterion for $k$--largeness.

\begin{lemma}\label{le:collaps}
Let $f\colon A\rightarrow B$ be a simplicial map from a flag simplicial complex $A$ onto a flag simplicial complex $B$. Suppose that vertices $a,a'$ of $A$ are adjacent if and only if the vertices $f(a), f(a')$ are adjacent or equal. Then $A$ is $k$--large if and only if $B$ is $k$--large.
\end{lemma}

Note that $f$ is a homotopy equivalence.

\begin{proof}
It suffices to show that the lengths of shortest full cycles in $A$ and $B$ are equal. A cycle $\beta$ in $B$ lifts to $A$. If $\beta$ has no diagonals, then neither does its lift. Conversely, suppose that we have a full cycle $\alpha$ in $A$. Then for non-consecutive vertices $a,a''$ of $\alpha$ the vertices $f(a),f(a'')$ are neither equal nor adjacent. In particular, if $a,a',a''$ are consecutive, then $f(a), f(a'')$ are not adjacent, hence $f(a)\neq f(a')$. Therefore $f(\alpha)$ is a cycle and has no diagonals.
\end{proof}

Here is another criterion.

\begin{defin}\label{def:gamma*}
Let $\Gamma$ be a graph whose maximal cliques (with respect to inclusion) intersect only along vertices. Denote by $V,M$ the sets of vertices and maximal
cliques of $\Gamma$. Consider the following graph $\Gamma^*$ with vertex
set $V\cup M$. We connect $v,v'\in V$ by an edge in $\Gamma^*$ if
they are connected by an edge in $\Gamma$. We connect $m,m'\in M$ if $m\cap m'\neq \emptyset$, and we connect $v\in
V, m\in M$ if $v\in m$. See Figure~\ref{fig_gstar}.
\end{defin}

\begin{figure}[h]
 \begin{center}
\resizebox{!}{0.4\textwidth}{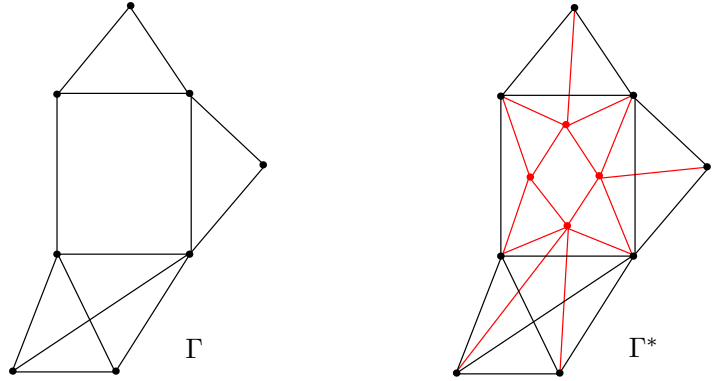}
\caption[stargraph]{An example of a graph $\Gamma$ and associated $\Gamma^*$.}
\label{fig_gstar}
\end{center}
\end{figure}

\begin{lemma}\label{le:gamma*}
Let $\Gamma$ be a graph whose maximal cliques
intersect only along vertices. The flag complex spanned on $\Gamma^*$ is $k$--large if and only if the flag complex spanned on $\Gamma$ is $k$--large.
\end{lemma}
\begin{proof}
The `only if' part is obvious, since the complex spanned on $\Gamma$ is a full subcomplex of the complex spanned on $\Gamma^*$. For the converse, let $\gamma^*$ be a shortest full cycle in $\Gamma^*$ with vertices $\rho_0,\rho_1,\ldots,\rho_n=\rho_0$. Assume
by contradiction $n<k$. We denote the relevant vertices of $\rho_i$ in
the following way. If $\rho_i$ is a vertex, then let
$v_{2i}=v_{2i+1}=\rho_i$. If $\rho_i$ is a clique, then let
$v_{2i}=\rho_i\cap\rho_{i-1}, v_{2i+1}=\rho_i\cap \rho_{i+1}$ (understood
cyclically).
Note that in this case $v_{2i}\neq v_{2i+1}$, since $\gamma^*$ has no
diagonals.

Let $\gamma$ be the closed edge-path formed by $(v_i)_{i=0}^{2n-1}$, after removing
consecutive repeating vertices.
Since $\gamma^*$ has no diagonals, there are combinatorially only three
possibilities for a triple of consecutive vertices $u,v,w$ of $\gamma$, up
to interchanging $u$ with $w$: Either
$u=\rho_{i-1},v=\rho_{i},w=\rho_{i+1}$ for some $i$, or
$v=\rho_i\cap\rho_{i+1}$ with $u\in \rho_i,w\in \rho_{i+1}$, or else
$v=\rho_i\in \rho_{i+1}$ with $u=\rho_{i-1}$. We claim that in all three
cases $u\neq w$ and moreover $u$ and $w$ are not connected by an edge. In
the first case this follows from the fact that $\rho_{i-1}\rho_{i+1}$ is
not a diagonal of $\gamma^*$. In the other two cases the clique
$\rho_{i+1}$ would have at least $vw$ in common with the triangle $uvw$,
so it would have to be equal to the maximal clique containing $uvw$. In
the third case this contradicts the fact that $\rho_{i-1}\rho_{i+1}$ is
not a diagonal of $\gamma^*$. In the second case we get that $\rho_{i}$ is
also the maximal clique containing $uvw$, so it coincides with
$\rho_{i+1}$, contradiction. This proves the claim, so that in particular
$\gamma$ is locally embedded.

Let $g\colon \gamma^*\rightarrow \gamma$ be a map defined in the
following way. If $\rho_i$ is a vertex, then let $g(\rho_i)=\rho_i$.
Otherwise, let $g(\rho_i)$ be the barycenter of the edge $v_{2i}v_{2i+1}$.
This extends to a simplicial (possibly degenerate) map $g$
between the barycentric subdivisions of $\gamma^*$ and $\gamma$. Hence
$|\gamma|\leq |\gamma^*|=n<k$. Since $\Gamma$ is $k$--large, $\gamma$ can be triangulated by consecutively adding diagonals. This
yields three consecutive vertices $u,v,w$ on $\gamma$ such that $uw$ is an
edge, and contradicts the claim.
\end{proof}

Finally, we have the following criterion.

\begin{defin}
\label{def:gammatilde}
Let $\Gamma=(V,E)$ be a graph of girth $\geq 4$. Let $\widetilde{\Gamma}$ be the following graph whose vertices are pairs $(v,\sigma)$, where $v\in V,\  \sigma\in V\cup E$ and $v\subset \sigma$. Vertices $(v,\sigma), (v',\sigma')$ are connected by an edge if $v=v'$ or $v$ and $v'$ are adjacent and $\sigma\in \{v,vv'\}, \sigma'\in \{v',vv'\}$ (see Figure~\ref{fig_gtilde}).
\end{defin}

\begin{figure}[h]
 \begin{center}
\resizebox{!}{0.4\textwidth}{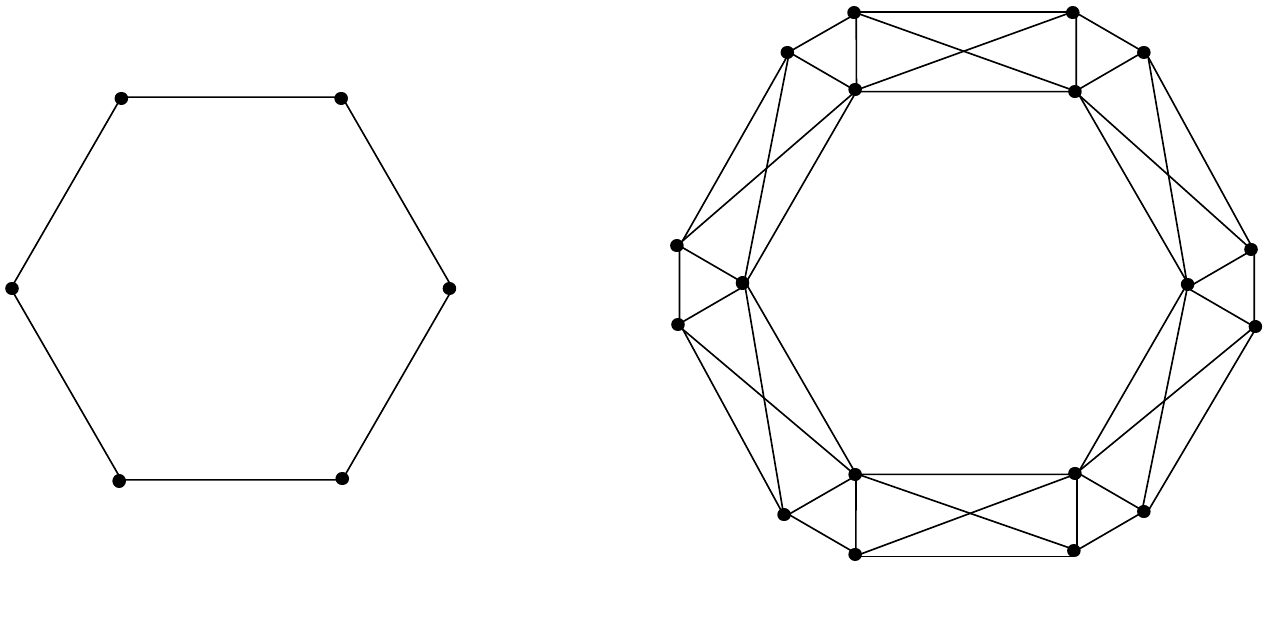}
\caption[stargraph]{An example of a graph $\Gamma$ and associated $\widetilde\Gamma$.}
\label{fig_gtilde}
\end{center}
\end{figure}

\begin{lemma}
\label{lem:gammatilde}
Let $\Gamma$ be a graph of girth $\geq 4$. The flag complex spanned on $\widetilde{\Gamma}$ is $k$--large if and only if $\Gamma$ has girth $\geq k$.
\end{lemma}
\begin{proof}
The `only if' part is again obvious, since $\Gamma$ embeds as a full subcomplex of the complex spanned on $\widetilde{\Gamma}$ under the map $v\rightarrow (v,v)$. For the converse, let $g\colon \widetilde{\Gamma}\rightarrow \Gamma$ be the simplicial map mapping each $(v,\sigma)$ to $v$.
Note that for each vertex $v$ of $\Gamma$, its preimage $g^{-1}(v)$ is a clique. The preimage $g^{-1}(e^\circ)$ of each open edge $e^\circ$ of $\Gamma$ is also contained in a clique (on $4$ vertices). If $\gamma$ is a cycle in $\widetilde{\Gamma}$ of length $<k$, then $g(\gamma)$ is homotopically trivial.
Hence $g(\gamma)$ backtracks in the sense that there are three consecutive vertices $vv'v''$ of $g(\gamma)$ with $v=v''$ or four consecutive vertices  $vv'v''v'''$ with $v=v''', v'=v''$. Considering $g^{-1}(v)$ in the first case and $g^{-1}(vv'^\circ)$ in the other produces a diagonal of $\gamma$.
\end{proof}

\section{The Coxeter group of type $(2,4,4)$ is not systolic}
\label{sec:244}

In this section we prove Theorem~\ref{thm:244} saying that the Coxeter group of type $(2,4,4)$ is not systolic.

\begin{proof}[Proof of Theorem~\ref{thm:244}]
Let $W$ be the $(2,4,4)$ triangle group with Coxeter generating set $s,t,r$, where $sr=rs$.
Let $g=rtst, \ h=tsrts$. Then $t$ conjugates $g^2$ and $h^2$. The subgroup $W'=\langle g,h\rangle$ is the Klein bottle group with relation $gh=hg^{-1}$. In particular, $W'$ is a torsion-free group that is virtually $\Z^2$.

We now follow the proof of \cite[Thm 4.1]{EP}. A \emph{systolic flat} $\E^2_{\triangle}$ is the systolic complex that is the equilateral triangulation of the Euclidean plane. Suppose that $W$ acts geometrically on a systolic complex $X$. By the systolic flat torus theorem \cite[Cor 6.2(1) and Thm 5.4]{Els}, the torsion-free subgroup $W'$ acts properly on a systolic flat $\E^2_{\triangle}\subset X$. If the Klein bottle group acts properly by isometries on the Euclidean plane, then $h$ acts as a glide reflection and $g$ acts as a translation in the direction perpendicular to the glide reflection axis. Since $\E^2_{\triangle}$ is equipped additionally with a combinatorial structure, there are only two possibilities for the axes of $g,h$. Exactly one of them is \emph{quasi-convex} in the sense that a $1$--skeleton geodesic starting and terminating at the axis is contained in its uniform neighborhood \cite[Prop 3.11]{E2}. By \cite[Prop 3.12]{E2}, this contradicts the fact that $g^2,h^2$ are conjugate. See the
proof of \cite[Thm 4.1]{EP} for details.
\end{proof}

\section{Coxeter complexes of type $(2,3,6)$}
\label{Sec:236}

To illustrate the method used in the proof of Theorem~\ref{thm:dim2}, we first show how to systolize the Coxeter complex $\Sigma$ of type $(2,3,6)$.

We say that a vertex is of \emph{type} $\bf{k}$ if its stabilizer is of order $2k$. The links of vertices of type $\bf{2}$ are squares preventing the complex from being systolic. In order to systolize $\Sigma$ we will add diagonals to all these squares, and verify that no new short full cycles are created.

\begin{figure}[htbp]
	\begin{center}
	\resizebox{!}{0.7\textwidth}{\includegraphics{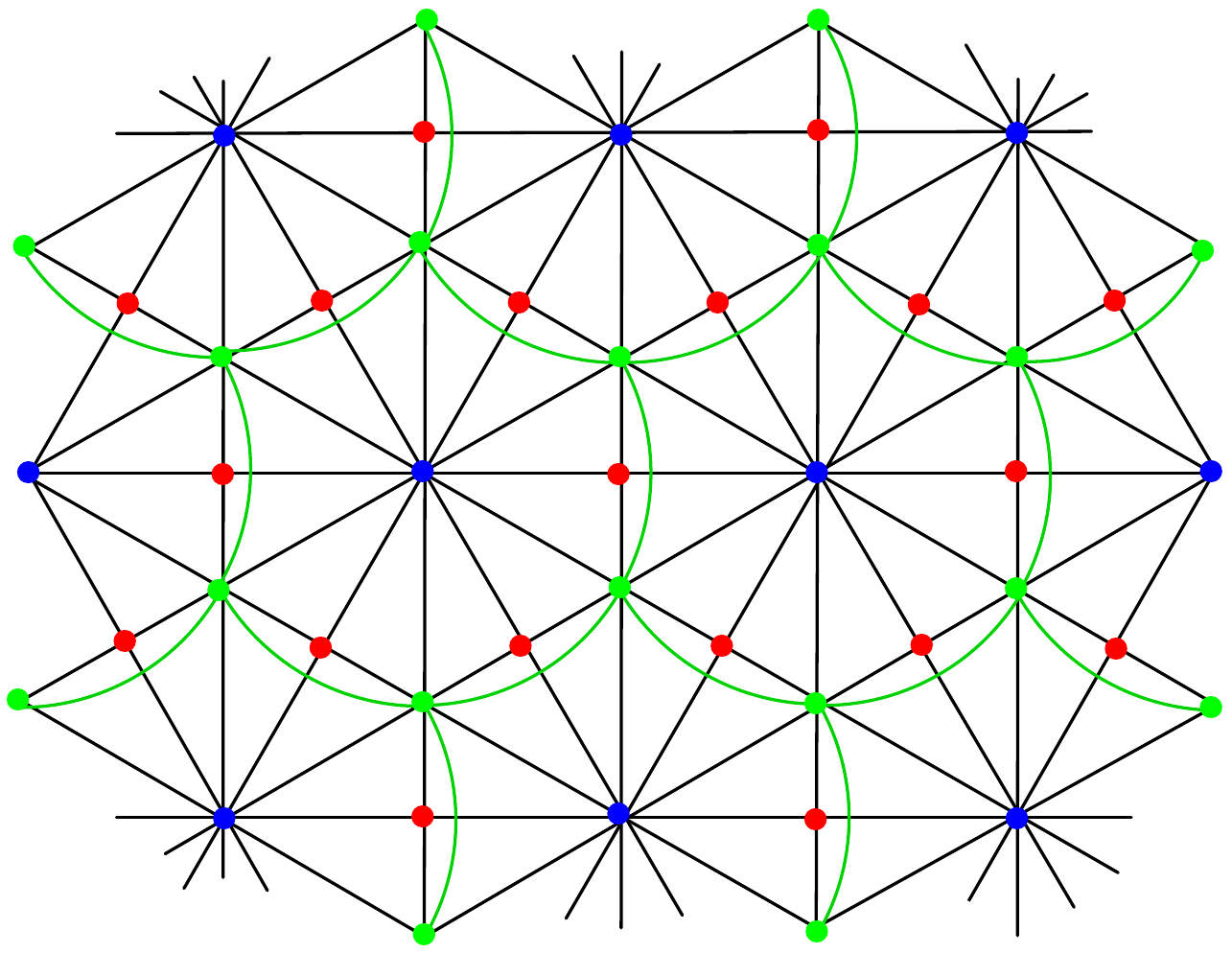}}
	\end{center}
	\caption[Systolized $(2,3,6)$]{Systolization of the Coxeter complex of type $(2,3,6)$.}
	\label{fig_236}
\end{figure}

The \emph{systolization} $\sys\Sigma$ of $\Sigma$ is the flag simplicial complex spanned on the following $1$--skeleton. The vertex set of $\sys{\Sigma}$ is the same as the vertex set of $\Sigma$. A pair of vertices is connected by an edge in $\sys{\Sigma}$ if it is either connected by an edge in $\Sigma$ or it is a pair of vertices of type $\bf 3$ that are adjacent to the same vertex of type $\bf 2$.

Figure~\ref{fig_236} shows the $1$--skeleton of $\sys{\Sigma}$. Vertices of type $\bf 2$ are pictured in red, vertices of type $\bf 3$ in green, vertices of type $\bf 6$ in blue. The new edges (green) are diagonals of the squares that are the links of the red vertices.

We now prove that $\sys\Sigma$ is indeed systolic. First observe that $\sys\Sigma$ is simply-connected, since each loop in the $1$--skeleton of $\sys\Sigma$ can be homotoped to a loop in the $1$--skeleton of $\Sigma$. It remains to show that the vertex links in $\sys{\Sigma}$ are $6$--large. The link of a vertex of type~$\bf 2$ in $\sys\Sigma$ is a pair of triangles glued along an edge, which is obviously $\infty$--large (since it is a join of an edge and a pair of vertices, it is a special case of Corollary~\ref{cor:join}).

The link $\Sigma_6$ of a (blue) vertex of type $\bf 6$ in $\Sigma$ is a cycle of length~$12$. Each of its edges joins a (green) vertex of type $\bf 3$ and a (red) vertex of type $\bf 2$. In the link $\sys\Sigma_6$ of a type $\bf 6$ vertex in $\sys\Sigma$, additional edges appear between pairs of green vertices originally at distance two in $\Sigma_6$. See the left side of Figure~\ref{fig_lk36}. The new edges create a cycle of length $6$ in $\sys\Sigma_6$. Observe that $\sys\Sigma_6$ is obtained from that cycle by gluing six triangles along edges. Hence to show that $\sys\Sigma_6$ is $6$--large, it suffices to apply six times Lemma~\ref{le:amalgam}.

\begin{figure}[h]
\begin{center}
 \begin{minipage}[b]{0.49\textwidth}
    	\resizebox{!}{0.18\textheight}{\includegraphics{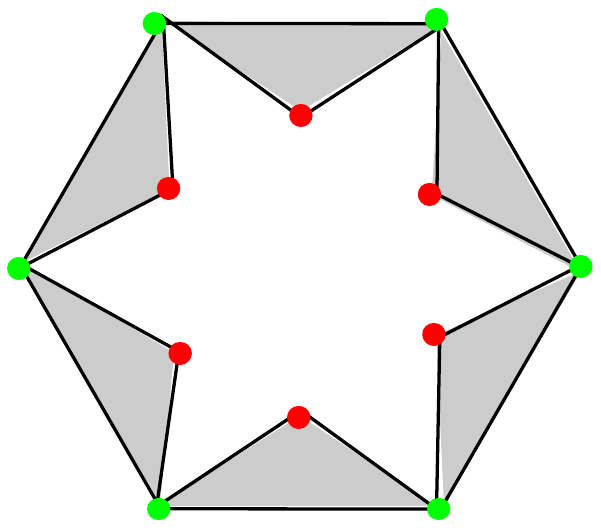}}
  \end{minipage}
  \begin{minipage}[b]{0.49\textwidth}
	\resizebox{!}{0.18\textheight}{\includegraphics{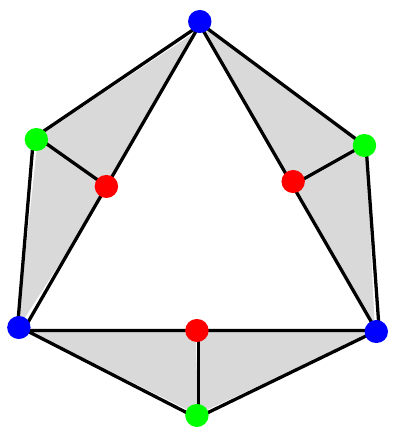}}
  \end{minipage}
  \caption[Links]{On the left the link of a vertex of type $\bf 6$ in the systolization of the Coxeter complex of type $(2,3,6)$. On the right the link of a type $\bf 3$ vertex.}
  \label{fig_lk36}
\end{center}
\end{figure}

The link $\Sigma_3$ of a (green) vertex of type~$\bf 3$ in $\Sigma$ is a blue-red cycle of length~$6$. This is the only vertex type whose link $\sys\Sigma_3$ in $\sys\Sigma$ contains new vertices. There is in $\sys\Sigma_3$ one additional (green) vertex of type~$\bf 3$ for each type~$\bf 2$ vertex, coning off the star of the latter, which is a blue-red-blue edge-path of length~$2$. See the right side of Figure~\ref{fig_lk36}.
There is a map $f\colon \sys\Sigma_3\rightarrow \Sigma_3$ satisfying the hypothesis of Lemma~\ref{le:collaps}. Since $\Sigma_3$ is $6$--large, $\sys\Sigma_3$ is $6$--large as well.

This concludes the proof that $\sys\Sigma$ is systolic and completes the discussion of our example. We could have also made $\Sigma$ systolic by removing vertices of type $\bf 2$ and edges of type $\bf {2-6}$, i.e.\ by merging pairs of chambers along edges of type $\bf {2-6}$. This approach might seem easier at first but does not generalize to buildings.

\section{Rank three}\label{Sec:dim2}

In this section we construct a systolization of the building in Theorem~\ref{thm:dim2}:

\begin{thm}\label{thm:dim2partCoxeter}
Let $W$ be a triangle Coxeter group with finite exponents distinct from $(2,4,4), (2,4,5)$ and $(2,5,5)$. Then the Coxeter realization of a building of type $W$ admits a systolization.
\end{thm}

We are postponing the discussion of a systolization of the Davis realization till Section~\ref{sec:Davis}.

If $W$ is finite, then a building $X$ of type $W$ is finite as well. Thus the simplex~$\sys X$ obtained by spanning simplices on all the vertex sets of $X$ is a systolization of~$X$. As Coxeter realizations of buildings of triangle type $(l,k,m)$ with $l,k,m\geq 3$ are themselves systolic, we only need to study triangle groups of type $(2,k,m)$ with $k\geq 3$ and $m\geq 6$.

Recall that a vertex is of \emph{type} $\bf{k}$ if its stabilizer is of order $2k$. Even if  $m = k$ we will distinguish these two types and will refer to corresponding vertices as of type $\bf m$ or $\bf k$, respectively.

\begin{constr}\label{def:sys2}
Let $X$ be the Coxeter realization of a building of triangle type $(2,k,m)$ with $k\geq 3$ and $m\geq 6$.
The \emph{systolization} $\sys X$ of $X$ is the flag simplicial complex spanned on the following $1$--skeleton. The vertex set of $\sys{X}$ is the same as the vertex set of $X$. A pair of vertices is connected by an edge in $\sys{X}$ if it is either connected by an edge in $X$ or it is a pair of vertices of type $\bf k$ that are adjacent to the same vertex of type $\bf 2$.
\end{constr}

Note that the inclusion $X\subset \sys{X}$ is obviously a quasi-isometry. Moreover, the action of the group of type preserving automorphisms of $X$ extends to~$\sys{X}$. Before we prove that $\sys{X}$ is systolic, we need a handful of lemmas. We consider the CAT(0) metric on $X$ in which the apartments are isometric to $\E^2$ or $\Hyp^2$. Unless mentioned otherwise, all stars are closed.

\begin{lemma}\label{le:convexStar}
Stars of vertices in $X$ are convex.
\end{lemma}
\begin{proof}
Let $v'\neq v$ be a vertex in $S=\st_X(v)$. Then all triangles of $S$ containing $v'$ share the edge $vv'$.
Since the triangles have all angles $\leq \frac{\pi}{2}$, this implies that $\lk_S(v')$ has diameter $\leq \pi$. Thus $S$ is locally convex at each $v'$. Consequently, $S$ is convex by \cite[Thm II.4.14]{BH} (see also \cite{BW}).
\end{proof}

\begin{cor}\label{cor:convexstar}
Let $v,v'$ be two vertices of type $\bf k$ adjacent to a vertex $w$ of type $\bf 2$. Then $w$ is a unique vertex of type $\bf 2$ adjacent to both $v$ and $v'$. Moreover, if $v,v'$ are both adjacent to a vertex $u$ of type $\bf m$, then $w$ and $u$ are adjacent.
\end{cor}

\begin{proof}
The concatenation $vv'$ of the edges $vw$ and $wv'$ is a geodesic, hence it determines $w$ uniquely.
By Lemma~\ref{le:convexStar} the geodesic $vv'$ is contained in the star of $u$. Since $w\in vv'$, we have that $w$ lies in the star of $u$.
\end{proof}

The final preparatory lemma involves triples of vertices of type $\bf k$. Below a \emph{fan} of triangles
at a vertex $v\in X$ is a subcomplex of $\st_X(v)$ that is a join of $v$ with a path graph in $\lk_X(v)$.

\begin{lemma}\label{le:2kloops}
In $X$ there is no cycle of length $6$ whose vertices have alternating types $\bf 2$ and $\bf k$.
\end{lemma}
\begin{proof}
Assume that there is such a cycle $vwv'w'v''w''v$, where the vertices $v,v',v''$ are of type $\bf k$ and $w,w',w''$ are of type $\bf 2$.
Let $S$ be the union of the star $\st_X(v)$ of $v$ in $X$ and the stars of the vertices of type $\bf 2$ in $\st_X(v)$.
We claim that $S$ is locally convex, hence convex by \cite[Thm II.4.14]{BH}. At a vertex of type $\bf k$ in $S$ that is distinct from $v$, all the triangles in $S$ have a common edge, so $S$ is locally convex as in the proof of Lemma~\ref{le:convexStar}. At a vertex $u\in S$ of type $\bf m$, there are fans of four triangles in $S$, but their corresponding angle equals $\frac{\pi}{m}$. Hence the diameter of $\lk_S(u)$ is $\frac{4\pi}{m}<\pi$, so~$S$ is locally convex at $u$, justifying the claim.
Thus the geodesic $v'v''$ lies in~$S$, whence $w'\in S$. All vertices of type $\bf 2$ in~$S$ are adjacent to $v$. Hence by Corollary~\ref{cor:convexstar} we have $w'=w=w''$, contradiction.
\end{proof}

We split the proof of Theorem~\ref{thm:dim2partCoxeter} into two steps. We first prove that $\sys{X}$ is simply-connected and then that its vertex links are $6$--large.

\begin{lemma}
\label{lem:simplycon}
The complex $\sys X$ is simply-connected.
\end{lemma}
\begin{proof}
The fundamental group of $\sys X$ is carried by its $1$--skeleton $\sys X^{(1)}$, we therefore only need to contract loops from $\sys X^{(1)}$.
For each edge $e=vv'$ in $\sys X$ connecting two vertices of type $\bf k$, there is an edge-path $\gamma$ in $X^{(1)}$ connecting $v,v'$ of combinatorial length two via a vertex of type $\bf 2$. The concatenation of $\gamma$ and $e$ bounds a triangle in $\sys X$, hence $e$ is homotopic to $\gamma$. Thus any loop in $\sys X^{(1)}$ is homotopic to a loop in $X^{(1)}$. Since $X$ is simply-connected, the complex $\sys X$ is simply-connected.
\end{proof}

\begin{lemma}
\label{lem:vertexlinks}
Links of vertices in $\sys X$ are $6$--large.
\end{lemma}
\begin{proof}
First consider the link $\sys X_2$ of a vertex of type $\bf 2$. The vertices of type~$\bf k$ in $\sys X_2$ are pairwise connected by edges and hence span a simplex. Thus $\sys X_2$ is a join of that simplex with a discrete set of vertices of type $\bf m$, and is hence $\infty$--large by Corollary~\ref{cor:join}.

Now consider the link $\sys X_m$ of a vertex of type $\bf m$. The star of each type $\bf 2$ vertex in $\sys X_m$ is a cone over the simplex formed by its adjacent vertices of type $\bf k$. Hence $\sys X_m$ is glued out of such cones and the subcomplex $V$ spanned by the vertices of type $\bf k$. By Lemma~\ref{le:amalgam} it suffices to show that $V$ is $6$--large. Suppose that there is a full cycle $\gamma$ in $V$ of length $4$ or $5$.
By Corollary~\ref{cor:convexstar}, for any pair of consecutive vertices $v, v'$ in $\gamma$ there is a unique vertex $w$ in $X_m$ of type $\bf 2$ adjacent to both $v$ and $v'$.
For any triple of consecutive vertices $v, v',v''$ in $\gamma$, consider the corresponding vertices $w,w'$ in $X_m$ of type $\bf 2$ forming an edge-path $vwv'w'v''$. Since $\gamma$ does not have diagonals, we have $w\neq w'$. Thus $\gamma$ gives rise to a locally embedded closed edge-path in $X_m$ of length $8$ or $10$. This contradicts the fact that the girth of $X_m$ is $2m$. Thus $V$ and $\sys X_m$ are $6$--large.

The link $\sys X_k$ of a vertex of type $\bf k$ contains all three types of vertices.
Each vertex $v$ of type $\bf k$ in $\sys X_k$ is adjacent to a unique vertex $w=w(v)$ of type $\bf 2$ in $\sys X_k$ (Corollary~\ref{cor:convexstar}) as well as to all type $\bf m$ neighbors of $w$. By Corollary~\ref{cor:convexstar}, the vertex $v$ is not adjacent to any other vertices of type $\bf m$ in~$\sys X_k$. By Lemma~\ref{le:2kloops}, two vertices $v,v'$ of type $\bf k$ in $\sys X_k$ are adjacent if and only if $w(v)=w(v')$. Hence the retraction $f\colon \sys X_k\rightarrow X_k$ assigning $f(v)=w(v)$ satisfies the hypothesis of Lemma~\ref{le:collaps}. Since $X_k$ is of girth $2k$, it is $6$--large, thus $\sys X_k$ is $6$--large as well.
\end{proof}

Thus the systolization $\sys{X}$ of the Coxeter realization of a building $X$ from Construction~\ref{def:sys2} is indeed systolic, as required in Theorem~\ref{thm:dim2partCoxeter}.

\section{Rank four}
\label{Sec_dim3}

In this section we construct a systolization of the Coxeter realization of a $3$--dimensional building required in Theorem~\ref{thm:dim3}:

\begin{thm}\label{thm:dim3partCoxeter}
Let $W$ be a Coxeter group of rank $4$ with finite exponents. Assume that all of its special rank $3$ subgroups are infinite and not of type $(2,4,4), (2,4,5)$ or $(2,5,5)$. Moreover, assume that there is at most one exponent $2$. Then the Coxeter realization of a building of type $W$ admits a systolization.
\end{thm}

Let $T=\bf{abcd}$ be the tetrahedron that is the base chamber of the Coxeter realization of a building $X$ of type $W$.
We label the edges of $T$ by the exponents in the Coxeter presentation. If all the exponents are $\geq 3$, then $X$ is systolic to begin with. Thus we further assume that there is precisely one edge labeled by $2$, say $\bf {ab}$.
Without loss of generality we can also assume that the edge $\bf{ac}$ is labeled by $m\geq 6$ and the edge $\bf{ad}$ is labeled by $k\geq 3$.
We have two possible labelings of $\bf{bc}$ and $\bf{bd}$, see Figure~\ref{fig_tethrahedron}.
\begin{itemize}
 \item[Case I.] The edge $\bf{bc}$ is labeled by $k'\geq 3$ and the edge $\bf{bd}$ is labeled by $m'\geq 6$.
 \item[Case II.] The edge $\bf{bc}$ is labeled by $m'\geq 6$ and the edge $\bf{bd}$ is labeled by $k'\geq 3$.
\end{itemize}
The edge $\bf{cd}$ is labeled by $l\geq 3$. Note that if, keeping the other conditions, one allowed more edges labeled by $2$, then it could only be the edge $\bf{cd}$, and the rest of the edges would be labeled as in case I. However, in our article we only allow the edge $\bf {ab}$ to be labeled by $2$.

\begin{figure}[h]
\begin{center}
\begin{minipage}[b]{0.49\textwidth}
	\resizebox{!}{0.16\textheight}{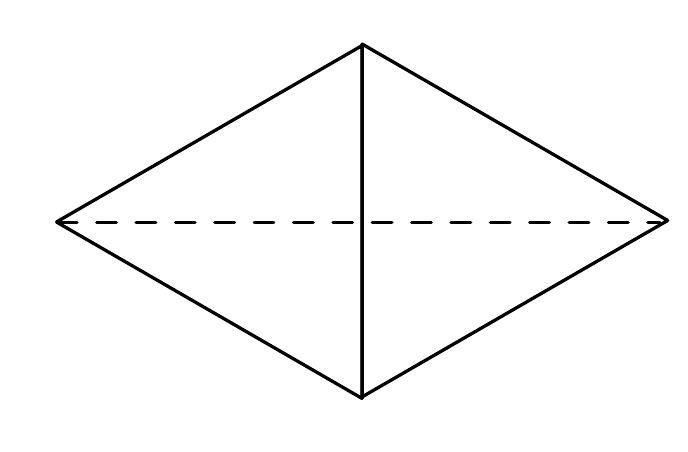}
  \end{minipage}
  \begin{minipage}[b]{0.49\textwidth}
    	\resizebox{!}{0.16\textheight}{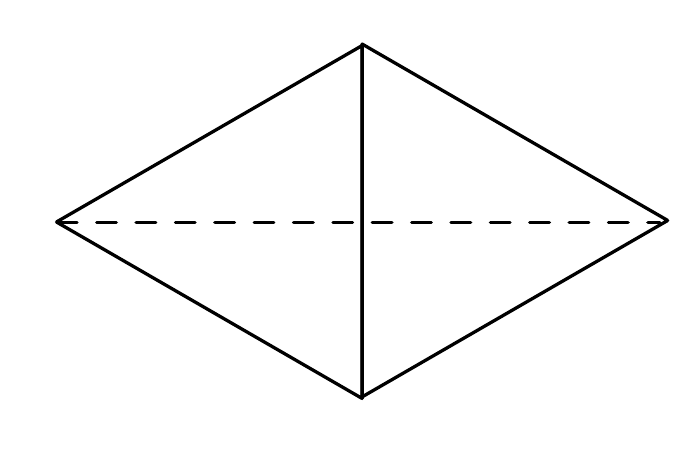}
\end{minipage}
\caption[tetrahedron]{The two cases of admissible tetrahedral types.}
\label{fig_tethrahedron}
\end{center}
\end{figure}

A simplex in $X$ is of \emph{type} $I\subset {\bf abcd}$ if it maps to $I$ under the retraction to~$T$. Unless otherwise mentioned we will denote vertices of type $\bf a$ by $a,a',a''$ etc. We say that two simplices are \emph{adjacent} if they span a simplex.

We now proceed with the construction of a systolization.

\begin{defin}\label{def:friends}
Two vertices of the same type $\bf c$ or $\bf d$ adjacent to a common edge of type $\bf{ab}$ are \emph{friends}.
Two vertices of the same type $\bf c$ or $\bf d$ that are not friends but adjacent to a common edge of type labeled by $k$ or $k'$ are \emph{acquaintances}.
\end{defin}

Note that in case II there are no acquaintances of type $\bf d$.

\begin{rem}
\label{rem:extra}
The link of an edge of type $\bf {ab}$ is a complete bipartite graph. Hence friends are also adjacent to common edges of types labeled by $k$ or $k'$, except for friends of type $\bf d$ in case II.
\end{rem}

\begin{constr}\label{def:sys3}
Let $X$ be the Coxeter realization of a building of the type described in case I or case II.
The \emph{systolization} $\sys X$ of $X$ is the flag simplicial complex spanned on the following $1$--skeleton. The vertex set of $\sys{X}$ is the same as the vertex set of $X$. Two vertices are adjacent in $\sys{X}$ if they are either adjacent in $X$ or are friends or acquaintances.
More explicitly, by Remark~\ref{rem:extra}, that means:
\begin{enumerate}[label={Case (**)}, leftmargin=*]
 \item[Case I:]
      vertices of type $\bf c$ adjacent to a common edge of type $\bf{ad}$, or
      vertices of type $\bf d$ adjacent to a common edge of type $\bf{bc}$.
 \item[Case II:]
      vertices of type $\bf c$ adjacent to a common edge of type $\bf{ad}$ or $\bf{bd}$, or
      vertices of type $\bf d$ adjacent to a common edge of type $\bf{ab}$.
\end{enumerate}
\end{constr}

In fact, if $k\geq 6$ or $k'\geq 6$, then fewer new edges would have done the job of systolizing $X$. To make the argument uniform we chose Construction~\ref{def:sys3} over a ``minimal'' one.

Our systolization of $X$ induces systolizations of its $2$--dimensional vertex links which are slightly thicker than the ones defined in Construction~\ref{def:sys2}. For example if $X$ has a vertex link that is of type $(2,3,6)$, then in Construction~\ref{def:sys3} we also add edges between pairs of vertices of type~$\bf 6$ adjacent to the same vertex of type $\bf 2$, and not only between such pairs of vertices of type $\bf 3$ as in Construction~\ref{def:sys2}.

Before we dive into the proof of Theorem~\ref{thm:dim3partCoxeter}, we first establish some preliminary lemmas on the combinatorial structure of $X$. As in Section~\ref{Sec:dim2}, we find it convenient to use a metric argument.
The tetrahedron~$T$ admits a Euclidean $\widetilde{A}_3$ metric, in which the dihedral angles at edges $\bf ab$ and $\bf cd$ are $\frac{\pi}{2}$, and the remaining dihedral angles are $\frac{\pi}{3}$.
In particular, the dihedral angle $\frac{\pi}{n}$ at an edge labeled by an exponent $i$ is $\geq \frac{\pi}{i}$, or equivalently $i\geq n$.
This equips $X$ with a complete geodesic metric \cite[Thm~I.7.19]{BH}. By Theorem~\ref{thm:app:CAT(0)} in the appendix, the vertex links of $X$ are CAT(1). Since $X$ is simply connected, by \cite[Thm~II.5.4]{BH} the metric on $X$ is CAT(0). We now discuss a criterion allowing to verify that a subcomplex of $X$ is convex.

A \emph{fan} of tetrahedra at an edge $e\in X$ is a subcomplex of $\st_X(e)$ that is a join of $e$ with a path graph in $\lk_X(e)$. The \emph{length} of the fan is the length of the path graph. An \emph{outer fan} of a subcomplex $Y\subset X$ at an edge $e\in Y$ is a fan whose path graph is disjoint from $\lk_Y(e)\subset \lk_X(e)$ except at its endpoints. Note that if $e$ is an edge of $Y$ with dihedral angle $\frac{\pi}{n}$, then $Y$ is locally convex at the interior of $e$ if and only if each outer fan of $Y$ at $e$ has length $\geq n$. We need the following strengthening of this property.

\begin{defin}
\label{def:stronglyedge}
A simplicial map $\varphi\colon Y\rightarrow X$ is \emph{strongly locally convex} at an edge $e\in Y$, if $\varphi_{|\st_Y(e)}$ is an embedding and all of the following conditions hold. We pull back the edge types and angles from $X$ to $Y$ via~$\varphi$.
\begin{itemize}
\item
If the dihedral angle at $e$ is $\frac{\pi}{n}$, then each outer fan of $\varphi\big(\st_Y(e)\big)\subset \st_X(\varphi(e))$ at $\varphi(e)$ has length $\geq n$.
\item
If $e$ is of type $\bf {cd}$, then each outer fan of $\varphi\big(\st_Y(e)\big)\subset \st_X(\varphi(e))$ at $\varphi(e)$ has length $\geq 3$.
\item
If $e$ is of type labeled by $m$ or $m'$, then each outer fan of $\varphi\big(\st_Y(e)\big)\subset \st_X(\varphi(e))$ at $\varphi(e)$ has length $\geq 7$.
\end{itemize}
\end{defin}

\begin{rem}
\label{rem: strong->convex}
Let $Y$ be a connected simplicial complex with connected vertex links. Suppose that $\varphi\colon Y\rightarrow X$ is strongly locally convex at all of its edges. Then $\varphi$ is an embedding and $\varphi(Y)$ is convex in $X$. Indeed, by \cite[Thm II.4.14]{BH} it suffices to show that $\varphi$ is a local isometry at every vertex $y\in Y$. Equivalently, by the argument for \cite[Prop 2.2]{BW}, it suffices to show that $\lk_Y(y)$ embeds in $\lk_X(y)$ and is $\pi$--convex. This follows from Theorem~\ref{thm:app:convex} in the appendix.
\end{rem}

\begin{lemma}\label{le:convexStar3}
Vertex and edge stars in $X$ are convex.
\end{lemma}
\begin{proof}
Let $e$ be an edge in $S=\st_X(v)$ that is not incident to $v$. Suppose that the type of $e$ is labeled by an exponent $i$ and the angle at $e$ equals~$\frac{\pi}{n}$. Then the outer fans of $\st_S(e)\subset X$ at $e$ have length $\geq 2i-2\geq i\geq n$, verifying the first item of Definition~\ref{def:stronglyedge}. If $e$ is of type $\bf cd$, then $i=l\geq 3$, verifying the second item. Finally, if $i=m$ or $m'$, then $i\geq 6$ and hence $2i-2\geq 10$, verifying the third item. Thus (the inclusion in $X$ of) $S$ is strongly locally convex at $e$. Consequently, $S$ is convex by Remark~\ref{rem: strong->convex}. The proof that edge stars are convex is analogous.
\end{proof}

\begin{cor}
\label{cor:friends}
\begin{enumerate}[label={\upshape(\roman*)}]
\item
Let $c,c'$ be friends of type $\bf c$ in the link of an edge $ab$ of type $\bf{ab}$ in $X$.
The vertices $a,b$ are unique of types $\bf{a,b}$ adjacent to both $c,c'$.
Any vertex of type~$\bf{d}$ adjacent to both $c,c'$ is adjacent to $a$ and~$b$.
\item
Let $d,d'$ be friends of type $\bf d$ in the link of an edge $ab$ of type $\bf{ab}$ in~$X$.
The vertices $a,b$ are unique of types $\bf{a,b}$ adjacent to both $d,d'$. Any vertex of type~$\bf{c}$ adjacent to both $d,d'$ is adjacent to $a$ and~$b$.
\end{enumerate}
\end{cor}
\begin{proof}
For assertion (i), note that the angle between the triangles $abc$ and $abc'$ is $\pi$, hence the geodesic $cc'$ passes through $ab$. If there is another vertex $a'$ of type $\bf a$ adjacent to both $c$ and $c'$, then by Lemma~\ref{le:convexStar3} the geodesic $cc'$ is contained in $\st_X(a')$. Thus $a\in \st_X(a')$, which is a contradiction. Remaining assertions follow in the same way.
\end{proof}

\begin{cor}
\label{cor:acquaint}
\begin{enumerate}[label={\upshape(\roman*)}]
\item
Let $c,c'$ be acquaintances of type $\bf c$ in the link of an edge $ad$ of type $\bf{ad}$ in $X$. The vertices $a,d$ are unique of types $\bf{a,d}$ adjacent to both $c,c'$. There is no vertex of type $\bf{b}$ adjacent to both $c,c'$.
\item
In case $\mathrm{I}$, let $d,d'$ be acquaintances of type $\bf d$ in the link of an edge $bc$ of type $\bf{bc}$ in $X$. The vertices $b,c$ are unique of types $\bf{b,c}$ adjacent to both $d,d'$. There is no vertex of type $\bf{a}$ adjacent to both $d,d'$.
\item
In case $\mathrm{II}$, let $c,c'$ be acquaintances of type $\bf c$ in the link of an edge $bd$ of type $\bf{bd}$ in $X$. The vertices $b,d$ are unique of types $\bf{b,d}$ adjacent to both $c,c'$. There is no vertex of type $\bf{a}$ adjacent to both $c,c'$.
\end{enumerate}
\end{cor}
\begin{proof}
For assertion (i), note that the angle between the triangles $adc$ and $adc'$ is $\geq 4\frac{\pi}{3}>\pi$. Thus the geodesic $cc'$ passes through $ad$. If there is another vertex $a'$ of type $\bf a$ adjacent to both $c$ and $c'$, then by Lemma~\ref{le:convexStar3} the geodesic $cc'$ is contained in $\st_X(a')$. Thus $a\in \st_X(a')$, which is a contradiction. Remaining assertions follow in the same way.
\end{proof}

We need the following classification of triples and quadruples of friends and acquaintances. The classification in cases~I and~II is similar, but the proof varies, so we split the statement according to cases~I and~II.

\begin{lemma}
\label{lem:cloopsI}
In case $\mathrm{I}$:
\begin{enumerate}[label={\upshape(\roman*)}]
\item
If $c,c',c''$ are pairwise friends or acquaintances of type $\bf c$, then there is an edge of type $\bf ad$ adjacent to all of $c,c',c''$. This edge is unique, unless $c,c',c''$ are friends adjacent to a common edge of type $\bf ab$.
\item
The same holds for a set of four vertices of type $\bf c$.
\item
If $d,d',d''$ are pairwise friends or acquaintances of type $\bf d$, then there is an edge of type $\bf bc$ adjacent to all of $d,d',d''$. This edge is unique, unless $d,d',d''$ are friends adjacent to a common edge of type $\bf ab$.
\item
The same holds for a set of four vertices of type $\bf d$.
\end{enumerate}
\end{lemma}

To prove Lemma~\ref{lem:cloopsI}, we first establish the following.

\begin{sublem}
\label{sub:I}
In case $\mathrm{I}$, let $c$ be a vertex of type $\bf c$ in $X$. Let $S$ be the union of the star $\st_X(c)$ of $c$ in $X$ and the stars $\st_X(ad)$ of all the edges $ad$ of type $\bf{ad}$ in $\lk_X(c)$. Then $S$ is convex.
\end{sublem}

\begin{proof}
Let $S_0$ be the simplicial complex obtained from $\st_X(c)$ by amalgamating along edges $ab$ of type $\bf ab$ in $\lk_X(c)$ all the triangles $ab\tilde{c}$ of type $\bf{abc}$ in $X$ with $\tilde{c}\neq c$. By Corollary~\ref{cor:friends}(i), the obvious map $\varphi_0\colon S_0\rightarrow X$ is an embedding.

By Corollaries~\ref{cor:acquaint}(i) and~\ref{cor:friends}(i), the star $\st_X(ad)$ of an edge $ad$ of type $\bf{ad}$ in $\lk_X(c)$ intersects $\varphi_0(S_0)$ along the union of $\st_X(adc)$ and triangles $ab\tilde{c}$ with $b$ adjacent to $ad$. Let $S_1$ be the simplicial complex obtained from $S_0$ by amalgamating with all $\st_X(ad)$ above along these sets of intersection. We denote the copy of $\st_X(ad)$ in $S_1$ by $S(ad)$.
We will prove that the obvious map $\varphi_1 \colon S_1\rightarrow X$ is strongly locally convex at each edge, and hence by Remark~\ref{rem: strong->convex} the map $\varphi_1$ is an embedding and $S=\varphi_1(S_1)$ is convex.

By the proof of Lemma~\ref{le:convexStar3}, it suffices to prove that $\varphi_1$ is strongly edge convex at every edge $e$ in~$S_0$. At an edge $e$ in $S_0$ of type $\bf{ab}$, by Remark~\ref{rem:extra} the map $\varphi_1$ is a local isomorphism. It is a local isomorphism as well at all $e$ incident to $c$ and all $e$ of type $\bf{ad}$.

For a different $e$ of type $I$, we compare the label~$i(e)$ on $I\subset \bf{abcd}$ with the length $f(e)$ of a longest fan at $e$ of tetrahedra in~$S_1$. An edge $e=bd$ in $S_0$ of type $\bf{bd}$ is contained in exactly these $S(ad)$ for which $a$ and $b$ are adjacent. Such $S(ad)$ paired with another $S(a'd)$ gives rise
to fans of length~$4$ at $e$. However, the label $i(e)=m'$ is~$\geq 6$. If $e$ is an edge of a triangle $ab\tilde{c}$ of type $\bf{abc}$ with $\tilde{c}\neq c$, then it is contained in exactly these $S(ad)$ for which $b$ and $d$ are adjacent. Here the longest fans are of length~$2$ at $e=b\tilde{c}$ and~$4$ at $e=a\tilde{c}$. However, at $\bf{ac}$ the label $i(e)=m$ is also $\geq 6$. Thus in all of the cases the outer fans of $\varphi_1\big(\st_{S_1}(e)\big)$ at $\varphi_1(e)$ have length $\geq 2i(e)-f(e)\geq i(e)$, or even $\geq 8$ for $i(e)=m,m'$, as desired.
\end{proof}

\begin{proof}[Proof of Lemma~\ref{lem:cloopsI}]
We begin by proving assertion (i).
By Sublemma~\ref{sub:I}, the geodesic $c'c''$ lies in $S$, whence the unique vertex $a$ from Corollary~\ref{cor:friends}(i) or~\ref{cor:acquaint}(i) adjacent to both $c',c''$ lies in $S$. All vertices of type $\bf a$ in $S$ lie in $\lk_X(c)$, thus $a$ is adjacent to $c$, as desired. If $c',c''$ are acquaintances, then we obtain the vertex $d$ adjacent to $a$ in the same way. Otherwise, if $c',c''$ are friends, then let $b$ be the unique vertex from Corollary~\ref{cor:friends}(i) adjacent to both $c',c''$. If $b$ is adjacent to $c$, then any vertex $d$ adjacent to $ab$ satisfies assertion (i) by Remark~\ref{rem:extra}. Otherwise, since $b$ lies in $S$, there is
a unique vertex $d$ in $\lk_X(c)$ such that $bad$ is a triangle (this uses the description of $S$ as the isomorphic image of $S_1$ from Sublemma~\ref{sub:I}). Such $d$ satisfies assertion~(i).

Assertion (iii) follows from the symmetry between the types $\bf {ad}$ and~$\bf {bc}$.

For assertion (ii), assume that there is a fourth vertex $c^*$ of type $\bf c$ that is a friend or acquaintance of all $c,c',c''$. The vertex $a$ is adjacent to $c^*$, as before. If any pair of the vertices of type $c$ are acquaintances, then the geodesic joining them passes through required $ad$. So now we assume that all the vertices of type $\bf c$ are pairwise friends. Let $d',d''$ be vertices of type $\bf d$ adjacent to the triples $c,c',c^*$ and $c,c'',c^*$ guaranteed by assertion (i). Denote the edges from Corollary~\ref{cor:friends}(i) for the pairs $cc^*,cc',cc''$ by $ab, ab', ab''$. By Corollary~\ref{cor:friends}(i), we have a loop $bd'b'db''d''b$. If some of the vertices $d,d',d''$ coincide, say $d=d'$, then $d$ satisfies the existence part of assertion (i).
If $d$ is not unique, then by the uniqueness part of assertion~(i) for triples $c,c',c^*$ and $c,c'',c^*$, the edge $ab$ is adjacent to both $c'$ and $c''$, as desired.
If the vertices $d,d',d''$ are distinct, then they satisfy the hypothesis of assertion~(iii). Hence $b=b'=b''$, so that all $c,c',c'',c^*$ are adjacent to $b$, whence to $d$ by Remark~\ref{rem:extra}.

Assertion (iv) follows from the symmetry between the types $\bf {ad}$ and~$\bf {bc}$.
\end{proof}

\begin{lemma}\label{lem:cloopsII}
In case $\mathrm{II}$:
\begin{enumerate}[label={\upshape(\roman*)}]
\item
If $c,c',c''$ are pairwise friends or acquaintances of type $\bf c$, then there is an edge of type $\bf ad$ or $\bf bd$ adjacent to all of $c,c',c''$. This edge is unique, unless $c,c',c''$ are friends adjacent to a common edge of type~$\bf ab$.
\item
The same holds for a set of four vertices of type $\bf c$.
\item
Sets of friends of type $\bf d$ are adjacent to a common edge of type $\bf ab$.
\end{enumerate}
\end{lemma}

\begin{sublem}
\label{sub:II}
In case $\mathrm{II}$:
\begin{enumerate}[label={\upshape(\roman*)}]
\item
Let $c$ be a vertex of type $\bf c$ in $X$. Let $S$ be the union of the star $\st_X(c)$ of $c$ in $X$ and the stars $\st_X(ad),\st_X(bd)$ of all the edges of types $\bf{ad},\bf{bd}$ in $\lk_X(c)$. Then $S$ is convex.
\item
Let $d$ be a vertex of type $\bf d$ in $X$. Let $S'$ be the union of the star $\st_X(d)$ of $d$ in $X$ and the stars $\st_X(ab)$ of all the edges of types $\bf{ab}$ in $\lk_X(d)$. Then $S'$ is convex.
\end{enumerate}
\end{sublem}

\begin{proof}
We begin by proving assertion (i). Let $S_0$ be the simplicial complex obtained from $\st_X(c)$ by amalgamating along $\st_X(abc)$ all $\st_X(ab)$ with $ab$ an edge of type $\bf ab$ in $\lk_X(c)$. By Corollary~\ref{cor:friends}(i), the obvious map $\varphi_0\colon S_0\rightarrow X$ is an embedding.

By Corollaries~\ref{cor:acquaint}(i) and~\ref{cor:friends}(i), the star $\st_X(ad)$ of an edge $ad$ of type $\bf{ad}$ in $\lk_X(c)$ intersects $\varphi_0(S_0)$ along the union of tetrahedra $ab\tilde{c}d$ with $b$ adjacent to $ad$. Analogously, the star $\st_X(bd)$ of an edge $bd$ of type $\bf{bd}$ in $\lk_X(c)$ intersects $\varphi_0(S_0)$ along the union of tetrahedra $ab\tilde{c}d$ with $a$ adjacent to $bd$.
Let $S_1$ be the simplicial complex obtained from $S_0$ by amalgamating with all $\st_X(xd)$ above along these sets of intersection.
We denote the copy of $\st_X(xd)$ in $S_1$ by $S(xd)$. It suffices to prove that the obvious map $\varphi_1 \colon S_1\rightarrow X$ is strongly locally convex at each edge.

By the proof of Lemma~\ref{le:convexStar3}, it suffices to prove that $\varphi_1$ is strongly locally convex at every edge $e$ in~$S_0$. At an edge $e$ in $S_0$ not incident to a $\tilde{c}\neq c$ of type $\bf c$, the map $\varphi_1$ is a local isomorphism. For a different $e$, incident to $\tilde{c}\neq c$, in a tetrahedron $ab\tilde{c}d$, we
proceed as in the proof of Sublemma~\ref{sub:I}, comparing $f(e)$ with $i(e)$. There are in $S_1$ fans of length~$4$ at $e=a\tilde{c}$ (respectively $e=b\tilde{c}$) coming from $S(ad), S(ad')$ (respectively $S(bd), S(bd')$), but in that case $i(e)\in\{m,m'\}$. There are fans of length $3$ at $e=\tilde{c}d$ coming from $S(ad),S(bd),$ but in that case $i(e)\geq 3$. Thus we can conclude using the same estimate as in the proof of Sublemma~\ref{sub:I}. This finishes the proof of assertion (i).

For assertion (ii), note that by Corollary~\ref{cor:friends}(ii), the subcomplex $S'$ is an amalgam of $\st_X(d)$ along $\st_X(abd)$ with all $\st_X(ab)$ for $ab$ an edge of type $\bf ab$ in $\lk_X(d)$. We prove that $S'\subset S$ is strongly locally convex by examining the edges $e$ in $\st_X(d)$. If $e$ is incident to $d$ or of type $\bf ab$, then it lies in the interior of $S'$ and there is nothing to prove. If $e$ is of type $\bf ac$ or $\bf bc$, then there are in $S'$ fans of length $4$ at $e$, but then $i(e)\in\{m,m'\}$, and we conclude as before.
\end{proof}

\begin{proof}[Proof of Lemma~\ref{lem:cloopsII}]
We begin by proving assertion (i). The geodesic $c'c''$ lies in $S$ by Sublemma~\ref{sub:II}(i). Thus if $c',c''$ are friends, then the edge $ab$ from Corollary~\ref{cor:friends}(i) lies in $S$. If $ab$ is adjacent to $c$, then there is nothing to prove. Otherwise, there is a unique triangle $bad$ with $ad$ or $bd$ in $\lk_X(c)$, as desired. By symmetry, we can now assume that all $c,c',c''$ are pairwise acquaintances. Let $x'd', x''d''$ be the edges from Corollary~\ref{cor:acquaint}(i,iii)
applied to the pairs $c,c'$ and $c, c''$, where $x',x''$ are of types $\bf a$ or $\bf b$. We then inspect using Remark~\ref{rem: strong->convex} that the union of $\st_X(c)$ with the triangles $c'x'd', c''x''d''$ is convex. Hence the geodesic $c'c''$ passes through both $x'd'$ and $x''d''$. Thus $d'=d'', x'=x''$, as desired.

For assertion (iii), if we consider friends $d,d',d'',\ldots $ of type $\bf d$, then the edge $ab$ from Corollary~\ref{cor:friends}(ii) lies in $\lk_X(d'')$ by Sublemma~\ref{sub:II}(ii).

For assertion (ii), assume that there is a fourth vertex  $c^*$ of type $\bf c$ that is a friend or acquaintance of all $c,c',c''$. If any pair of the vertices of type $c$ are acquaintances, then the geodesic joining them passes through required $xd$. If all the vertices of type $c$ are pairwise friends, then denote the edges from Corollary~\ref{cor:friends}(i) for the pairs $cc^*,cc',cc''$ by $ab, a'b', a''b''$. By assertion~(i) applied to $c,c',c''$, we have $a'=a''$ or $b'=b''$. Similarly, $a=a'$ or $b=b'$, and $a=a''$ or $b=b''$. Thus $a=a'=a''$ or $b=b'=b''$, say that the former equality holds. We consider the vertices $d,d',d''$ of type $\bf d$ for the triples $c,c',c''; c,c',c^*$ and $c,c'',c^*$ guaranteed by assertion (i).
By Corollary~\ref{cor:friends}(i), we have a loop $bd'b'db''d''b$. If some of $d,d',d''$ coincide, then we conclude as in the proof of Lemma~\ref{lem:cloopsI}. If they are distinct, then by assertion (iii) we have $b=b'=b''$. By Remark~\ref{rem:extra} any $d$ adjacent to $ab$ satisfies the lemma.
\end{proof}

To prove Theorem~\ref{thm:dim3partCoxeter} we need to show that $\sys X$ is simply-connected and that all its vertex links are $6$--large. To prove that vertex links are $6$--large it suffices to prove that they are systolic. Equivalently, they are simply-connected and edge links of $\sys X$ are $6$--large. We first prove simple-connectedness of $\sys X$ and its vertex links and then embark on proving  $6$--largeness of the edge links in Section~\ref{sec:edge}.

\begin{lemma}
\label{lem:simplyconvertex}
The complex $\sys X$ and its vertex links are simply-connected.
\end{lemma}

\begin{proof}
All the edges that we add to the $1$--skeleton of $X$ to obtain the $1$--skeleton of $\sys X$ connect vertices at distance $2$ in $X$. Hence $\sys X$ is simply-connected by the same argument as in the proof of Lemma~\ref{lem:simplycon}: all loops can be homotoped into $X$ that is simply-connected.
Similarly, a link in $\sys X$ of a vertex of type $\bf {a}$ or $\bf b$ is obtained from its link in $X$ by adding edges between some vertices at distance $2$, by Corollaries~\ref{cor:friends} and~\ref{cor:acquaint}. The link $\lk_{\sys X}(c)$ of a vertex $c$ of type $\bf {c}$ is obtained from $\lk_{X}(c)$ in the following way. We add edges as before between some vertices at distance $2$. But we also add vertices of type $\bf c$ and some incident edges. However, by Lemmas~\ref{lem:cloopsI}(i) and~\ref{lem:cloopsII}(i), an edge between type $\bf c$ vertices is homotopic in $\lk_{\sys X}(c)$ to an edge-path of length $2$ through a vertex of type $\bf{a}$ (or possibly $\bf b$ in case II). Moreover, each edge-path of length two in $\lk_{\sys X}(c)$ whose only vertex of type $\bf c$ is the middle vertex, is homotopic by Corollaries~\ref{cor:friends}(i) and~\ref{cor:acquaint}(i,iii) to an edge-path of length $2$ with the middle vertex replaced by a vertex of type $\bf{a}$ (or possibly $\bf b$). Hence any loop in $\
lk_{\sys{X}}(c)$ can be homotoped into $\lk_{X}(c)$, which is simply-connected. Analogously, the link of a vertex of type $\bf {d}$ in $\sys X$ is simply-connected as well.
\end{proof}

\section{Edge links}
\label{sec:edge}

In this section we will prove that the links of edges of types $\bf{ab, ac, ad}$ and $\bf{cd}$ in $\sys X$ as well as links of edges of friends and acquaintances are $6$--large. Using symmetries of $T$ it will follow that the links of edges of types $\bf{bd}$ and $\bf{bc}$ are also $6$--large.
This will complete the proof of Theorem~\ref{thm:dim3partCoxeter}.

\begin{prop}
\label{lem:linkab}
The link of an edge of type $\bf{ab}$ is a simplex.
\end{prop}
\begin{proof}
By Remark~\ref{rem:extra}, the link of an edge $ab$ of type $\bf{ab}$ in $X$ is a complete bipartite graph on vertices of types $\bf c$ and $\bf d$.
The link $\lk_{\sys X}(ab)$ of $ab$ in $\sys X$ has the same set of vertices. All the vertices of type $\bf c$ in $\lk_{\sys X}({ab})$ are friends, and the same holds for all the vertices of type $\bf d$.
\end{proof}

\begin{prop}
The link of an edge of type $\bf{ad}$ is $\infty$--large.
\end{prop}
\begin{proof}
The link $\lk_X(ad)$ of an edge $ad$ of type $\bf{ad}$ is a
bipartite graph of girth $2k$ on vertices of types $\bf b$ and
$\bf c$. We will now describe how one obtains $\lk_{\sys
X}(ad)$ from $\lk_X(ad)$. By Construction~\ref{def:sys3}, all
the vertices of type $\bf c$ become connected by an edge. Each
vertex $d'$ of type $\bf d$ that appears in $\lk_{\sys X}(ad)$
is a friend of $d$, by Corollary~\ref{cor:acquaint}(ii). By
Corollary~\ref{cor:friends}(ii), there is a unique vertex
$b(d')$ of type $\bf b$ in $\lk_{\sys X}(ad)$ adjacent to $d'$.
Moreover, the vertices of type $\bf c$ in $\lk_{\sys X}(ad)$
that are adjacent to $d'$ are exactly the ones adjacent to
$b(d')$. Finally, by Lemmas~\ref{lem:cloopsI}(iii)
and~\ref{lem:cloopsII}(iii), vertices $d',d''$ in $\lk_{\sys
X}(ad)$ are adjacent if and only if $b(d')=b(d'')$. See
Figure~\ref{fig_ad}.

\begin{figure}[h]
\begin{center}
    	\resizebox{!}{0.25\textheight}{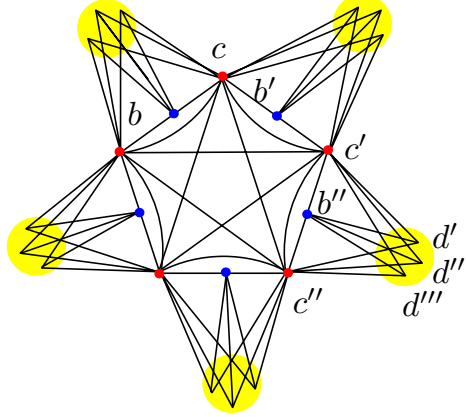}
\caption[ad-Link]{$\lk_{\sys X}(ad)$: yellow areas mark cliques of vertices $d'$ with common $b(d')$. }
\label{fig_ad}
\end{center}
\end{figure}

Let $B\subset \lk_{\sys X}(ad)$ be the subcomplex spanned by the vertices of type $\bf b$ and $\bf c$. Then the retraction $f\colon\lk_{\sys X}(ad)\rightarrow B$ defined by $f(d')=b(d')$ satisfies the hypothesis of Lemma~\ref{le:collaps}. Hence it suffices to prove that $B$ is $\infty$--large.
Observe that $B$ is glued out of simplices as follows. There is one base-simplex spanned by all the vertices of type $\bf c$ in $B$. Each
vertex of type $\bf b$ cones off a subsimplex of that base-simplex. Therefore by repeated application of Lemma~\ref{le:amalgam}, the complex $B$ is $\infty$--large.
\end{proof}

\begin{prop}
The link of an edge of friends is $\infty$--large.
\end{prop}
\begin{proof}
We only verify the proposition for a pair of friends of type $\bf c$, since the proof for type $\bf d$ friends is analogous.
Let $c,c'$ be friends. Let $a,b$ be the vertices from Corollary~\ref{cor:friends}(i), which are unique of type $\bf {a,b}$ in $\lk_{\sys X}(cc')$. Then the set $D$ consisting of $a,b$ and all the vertices of type $\bf d$ in $\lk_{\sys X}(cc')$ spans a simplex. This simplex is contained in another obtained by adding the set $C$ of vertices of type $\bf c$ adjacent to $ab$. Consider a vertex $c''$ of type $\bf c$ in $\lk_{\sys X}(cc')-C$. By Lemmas~\ref{lem:cloopsI}(i) and~\ref{lem:cloopsII}(i), there are unique $x\in \{a,b\}$ and $d$ of type $\bf d$ in $\lk_{\sys X}(cc')$ such that $c''$ is adjacent to $xd$ but not to $\{a,b\}-\{x\}$ or to other vertices of type $\bf d$. Note that $c''$ is adjacent to all the vertices of $C$. Now suppose that $c^*$ is another vertex of type $\bf c$ in $\lk_{\sys X}(cc')-C$. By Lemmas~\ref{lem:cloopsI}(ii) and~\ref{lem:cloopsII}(ii), vertices $c''$ and $c^*$ are adjacent if and only if $c^*$ is also adjacent to $xd$. Thus $\lk_{\sys X}(cc')$ is obtained from the simplex
spanned by $D\cup C$ by amalgamating with simplices
corresponding to edges $xd$ along subsimplices spanned by $\{x,d\}\cup C$. See Figure~\ref{fig_ccfriends}.
By Lemma~\ref{le:amalgam}, the link $\lk_{\sys X}(cc')$ is $\infty$--large.
\end{proof}

\begin{figure}[h]
\begin{minipage}[b]{0.49\textwidth}
	\begin{center}
	\resizebox{!}{0.22\textheight}{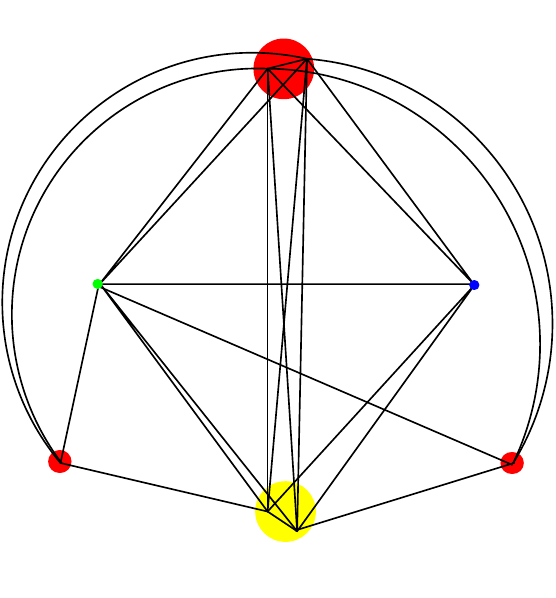}
	\end{center}
  \end{minipage}
  \begin{minipage}[b]{0.49\textwidth}
	\resizebox{!}{0.22\textheight}{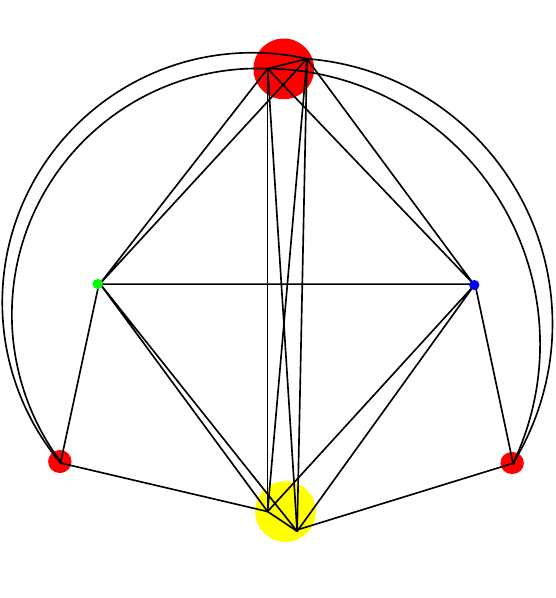}
\end{minipage}
\caption[cc-Link]{On the left the link $\lk_{\sys X}(cc')$ for friends $c,c'$ in case I, on the right in case II.}
\label{fig_ccfriends}
\end{figure}

\begin{prop}
The link of an edge of acquaintances is a simplex.
\end{prop}
\begin{proof}
Without loss of generality, we can assume that the acquaintances are of type $\bf c$, denote them by $c$ and $c'$. Again without loss of generality, let $a,d$ be the vertices from Corollary~\ref{cor:acquaint}(i). Then the link $\lk_{\sys X}(cc')$ contains $ad$ but no other vertex of type $\bf a$ or $\bf d$ or a vertex of type $\bf b$. Moreover, by Lemmas~\ref{lem:cloopsI}(i) and~\ref{lem:cloopsII}(i), all of the vertices of type $\bf c$ in $\lk_{\sys X}(cc')$ are connected to $ad$, hence to each other. Thus $\lk_{\sys X}(cc')$ is a simplex.
\end{proof}

\begin{prop}
The link of an edge of type $\bf{ac}$ is $6$--large.
\end{prop}
\begin{proof}
The link $\lk_X(ac)$ of an edge $ac$ of type $\bf{ac}$ is a bipartite graph of girth $2m$ on vertices of types $\bf b$ and $\bf d$.
When we pass to $\lk_{\sys X}(ac)$, two vertices of type $\bf d$ are connected by an edge if and only if they have a common neighbor of type $\bf b$, by Corollaries~\ref{cor:friends}(ii) and~\ref{cor:acquaint}(ii). Denote by $\Gamma\subset \lk_{\sys X}(ac)$ the subgraph spanned by the vertices of type $\bf d$.

There are two classes of vertices of type $\bf c$ in $\lk_{\sys X}(ac)$ coming from friends and acquaintances of $c$, respectively.
If $c'$ in $\lk_{\sys X}(ac)$ is a friend of $c$, then by Corollary~\ref{cor:friends}(i) it is adjacent to a unique vertex $b(c')$ of type $\bf b$ in $\lk_{\sys X}(ac)$ and to all its neighbors of type $\bf d$. Moreover, $c'$ is not adjacent to any other vertex of type $\bf d$ in $\lk_{\sys X}(ac)$. All $c'$ with common $b=b(c')$ span a clique, which will be called the \emph{$b$--clique}.

Assume now that $c'$ in $\lk_{\sys X}(ac)$ is an acquaintance of $c$. By Corollary~\ref{cor:acquaint}(i), the vertex $c'$ is adjacent to a unique vertex $d(c')$ of type $\bf{d}$ in $\lk_{\sys X}(ac)$, and to none of the vertices of type $\bf b$. All $c'$ with common $d=d(c')$ span a clique, which we call the \emph{$d$--clique}. By Lemmas~\ref{lem:cloopsI}(i) and~\ref{lem:cloopsII}(i), two vertices $c',c''$ of type $\bf c$ in $\lk_{\sys X}(ac)$ are adjacent if and only if they have a common neighbor of type $\bf d$. Thus if they are not in the same $b$--clique or $d$--clique, they are adjacent if and only if: either both $c',c''$ are friends of $c$ and the stars of $b(c'), b(c'')$ in $\lk_{X}(ac)$ intersect, or if $c'$ is friend of $c$ and $c''$ is an acquaintance of $c$ and $d(c'')$ is adjacent to $b(c')$, or vice-versa. See Figure~\ref{fig_ac1}.

\begin{figure}[h]
 \begin{center}
\resizebox{!}{0.4\textwidth}{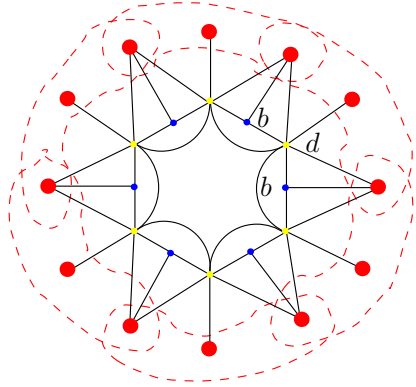}
\caption[ac-Link]{$\lk_{\sys X}(ac)$ in cases I and II. The big red dots stand for $b$--cliques and $d$--cliques. The edges connecting these cliques to other vertices represent multiple edges.}
\label{fig_ac1}
\end{center}
\end{figure}

The stars of all the vertices of type $\bf b$, and of acquaintances of $c$ in $\lk_{\sys X}(ac)$ are simplices. By Lemma~\ref{le:amalgam}, to prove that the link $\lk_{\sys X}(ac)$ is $6$--large it suffices to prove that the complex obtained by removing these vertices is $6$--large. We remove them, and denote by $\Gamma^*$ be the graph obtained from the $1$--skeleton of the resulting complex by collapsing each $b$--clique to a vertex.
By Lemma~\ref{le:collaps}, it suffices to verify that $\Gamma^*$ is $6$--large. The maximal cliques of $\Gamma$ correspond to $b$--cliques, thus $\Gamma^*$ is obtained from $\Gamma$ as in Definition~\ref{def:gamma*}. As in the proof of Lemma~\ref{lem:vertexlinks}, the flag complex spanned on $\Gamma$ is $6$--large. Hence by Lemma~\ref{le:gamma*}, the flag complex spanned on $\Gamma^*$ is $6$--large.
\end{proof}

\begin{prop}
The link of an edge of type $\bf{cd}$ is $6$--large.
\end{prop}

\begin{proof}
The link $\Gamma$ of an edge $cd$ in $X$ of type $\bf {cd}$ is a bipartite graph on vertices of type $\bf a$ and $\bf b$ of girth $2l\geq 6$.
We will now describe how one obtains $\lk_{\sys X}(cd)$ from $\Gamma$. If $c'$ is a vertex of type $\bf c$ in $\lk_{\sys X}(cd)$ that is a friend of $c$, then by Corollary~\ref{cor:friends}(i) it is adjacent to unique adjacent $a(c'),b(c')$ of type $\bf a,b$ in $\lk_{\sys X}(cd)$.
If $c'$ is an acquaintance of $c$, then by Corollary~\ref{cor:acquaint}(i) it is adjacent to a unique $x=x(c')$ in $\lk_{\sys X}(cd)$ of type $\bf a$, or possibly $\bf b$ in case~II.

Let $c',c''$ be of type $\bf c$ in $\lk_{\sys X}(cd)$. In case I the vertices $c',c''$ are adjacent if and only if $a(c')=a(c'')$, by Lemma~\ref{lem:cloopsI}(i). In case II, by Lemma~\ref{lem:cloopsII}(i), the vertices $c',c''$ are adjacent also if $b(c')=b(c'')$. In the same way we describe vertices of type $\bf d$ in $\lk_{\sys X}(cd)$. If vertices $c',d'$ of types $\bf c,d$ in $\lk_{\sys X}(cd)$ are adjacent, then by Corollary~\ref{cor:acquaint}(i--iii), the vertices $c,c'$ are friends, $d,d'$ are friends and $a(c')=a(d'), \ b(c')=b(d')$. Conversely, if we have latter equalities, then $c',d'$ are adjacent.

\begin{figure}[h]
\begin{center}
    	\resizebox{!}{0.4\textwidth}{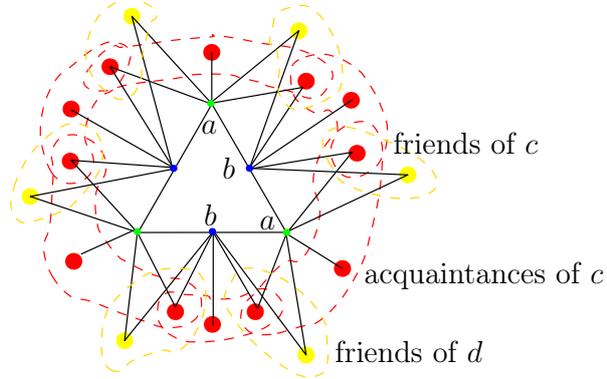}
\caption[cd-Link]{$\lk_{\sys X}(cd)$ in case II. The big red and yellow dots represent cliques of vertices of type $\bf c$, respectively $\bf d$. The encircled dots form even larger cliques.}
\label{fig_cd2}
\end{center}
\end{figure}

Assume first that we are in case II. All the vertices of type $\bf d$ in $\lk_{\sys X}(cd)$ are then friends of $d$, and their stars are simplices.
The stars of the acquaintances of $c$ are also simplices. See Figure~\ref{fig_cd2}. By Lemma~\ref{le:amalgam}, to prove that $\lk_{\sys X}(cd)$ is $6$--large it suffices to prove that the complex obtained by removing the vertices of type $\bf d$ and acquaintances of $c$ is $6$--large. We collapse the resulting complex by identifying all the friends $c'$ of $c$ with common $a(c')$ and $b(c')$. The result $\Gamma^*$ of the collapse is $6$--large by Lemma~\ref{le:gamma*}, since $\Gamma$ is $6$--large.

Finally, consider case I, see Figure~\ref{fig_cd1}. As before, we remove from $\lk_{\sys X}(cd)$ all the acquaintances of $c$ and $d$, using Lemma~\ref{le:amalgam}. Let $\widetilde{\Gamma}$ be the graph obtained by collapsing the vertices $c'$ of type $\bf c$ with common $a(c')$ and $b(c')$, and the vertices $d'$ of type $\bf d$ with common $a(d')$ and $b(d')$. Assign to each vertex $x$ of type $\bf a$ or $\bf b$ of $\widetilde{\Gamma}$ the pair $(x,x)$, and to each collapsed clique the pair $\big(a(c'),a(c')b(c')\big)$ for $c'$ of type $\bf c$, and $\big(b(d'),a(d')b(d')\big)$ for $d'$ of type $\bf d$. This shows that $\widetilde{\Gamma}$ has the form required in Definition~\ref{def:gammatilde}. By Lemma~\ref{lem:gammatilde}, the complex spanned on $\widetilde{\Gamma}$ is $6$--large, since $\Gamma$ has girth $\geq 6$.
\end{proof}

\begin{figure}[h]
\begin{center}
    	\resizebox{!}{0.4\textwidth}{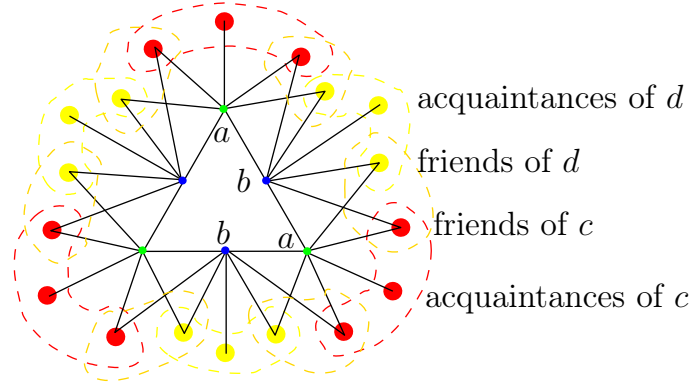}
\caption[cd-Link]{$\lk_{\sys X}(cd)$ in case I: The big red and yellow dots represent cliques of vertices of type $\bf c$, respectively $\bf d$. The encircled dots form cliques as well.}
\label{fig_cd1}
\end{center}
\end{figure}

This concludes the proof that the systolization $\sys X$ in Construction~\ref{def:sys3} is indeed systolic, as required in Theorem~\ref{thm:dim3partCoxeter}.

\section{Systolization of the Davis realization}
\label{sec:Davis}

In this section we complete the proofs of Theorems~\ref{thm:dim2} and~\ref{thm:dim3}, by systolizing the Davis realization:

\begin{thm}\label{thm:dim2part2}
Let $W$ be a triangle Coxeter group with finite exponents distinct from $(2,4,4), (2,4,5)$ and $(2,5,5)$. Then the Davis realization of a building of type $W$ admits a systolization.
\end{thm}

\begin{thm}\label{thm:dim3part2}
Let $W$ be a Coxeter group of rank $4$ with finite exponents. Assume that all of its special rank $3$ subgroups are infinite and not of type $(2,4,4), (2,4,5)$ or $(2,5,5)$. Moreover, assume that there is at most one exponent $2$. Then the Davis realization of a building of type $W$ admits a systolization.
\end{thm}

\begin{defin}
Let $X$ be a simplicial complex. The \emph{face complex} $X^f$ of $X$ is the following simplicial complex. The vertex set of $X^f$ is the set of simplices of $X$. A set of vertices of $X^f$ spans a simplex of $X^f$, if the corresponding simplices of $X$ are all contained in a common simplex of $X$.
\end{defin}

Haglund observed the following.

\begin{prop}[{\cite[Prop B.1]{JS3}}]
\label{lem:face}
If $X$ is $k$--large, then its face complex $X^f$ is $k$--large. Consequently, if $X$ is systolic, then $X^f$ is systolic, since it admits a deformation retraction to $X$.
\end{prop}

\begin{proof}[Proof of Theorem~\ref{thm:dim2part2}]
Let $X$ be a building of rank $3$ type as in Theorem~\ref{thm:dim2part2}. If $W$ is finite, then the Davis realization of $X$ is finite as well, and
the simplex obtained by spanning simplices on all the vertex sets of the Davis realization is its systolization. Otherwise, the Davis realization of $X$ is the barycentric subdivision $X'$ of $X$. Let $\sys X$ be the systolization of $X$ from Construction~\ref{def:sys2}, except that in the case where all the exponents in the Coxeter presentation are $\geq 3$, we take $\sys X=X$. Let $\sys X^f$ be the face complex of $\sys X$. The embedding $X\subset\sys X$ induces an embedding $X^f\subset\sys X^f$, and composing with $X'\subset X^f$ we obtain a natural embedding $X'\subset \sys X^f$, which is a quasi-isometry. By Proposition~\ref{lem:face}, the face complex $\sys X^f$ is a systolization of $X'$.
\end{proof}

\begin{proof}[Proof of Theorem~\ref{thm:dim3part2}]
Let $X$ be a building of rank $4$ type as in Theorem~\ref{thm:dim3part2}. The Davis realization $X_D$ of $X$ is the $2$--dimensional subcomplex of the barycentric subdivision $X'$ of $X$ obtained by removing the open stars of all the vertices of $X'$ corresponding to the vertices of $X$. Let $\sys{X}$ be the systolization from Construction~\ref{def:sys3}, except that in the case where all the exponents in Coxeter presentation are $\geq 3$, we take $\sys X=X$. Let $\sys X^f$ be the face complex of $\sys X$, and let $\sys X^f_D$ be the subcomplex of $\sys X^f$ obtained by removing the open stars of all the vertices corresponding to the vertices of $X$. We will prove that $\sys X^f_D$ is a systolization of $X_D$.

By Proposition~\ref{lem:face}, the face complex $\sys X^f$ is systolic, whence all the vertex links of its full subcomplex $\sys X^f_D$ are $6$--large. To prove that $\sys X^f_D$ is systolic it remains to prove that it is contractible. We first claim that the link in $\sys X^f$ of any closed simplex in $\sys X^f-\sys X^f_D$ (which comes from a simplex $\sigma$ of $X$) is contractible. If $\sigma$ is not a vertex, then its link in $\sys X^f$ is a cone, coned off by the vertex corresponding to $\sigma$, thus it is contractible. If $\sigma$ is a vertex, then its link $\lk_{\sys X^f}(\sigma)$ can be collapsed as in Lemma~\ref{le:collaps} to $\big(\lk_{\sys X}(\sigma)\big)^f$. We showed in Lemma~\ref{lem:simplyconvertex} that $\lk_{\sys X}(\sigma)$ is simply-connected, and since it is $6$--large, it is systolic. Thus by Proposition~\ref{lem:face} its face complex $\big(\lk_{\sys X}(\sigma)\big)^f$ is systolic, in particular contractible. Thus $\lk_{\sys X^f}(\sigma)$ is contractible,
justifying the claim. In view of the claim, $\sys X^f_D$ is obtained from $\sys X^f$ by repeatedly removing open stars of simplices with contractible links (starting from maximal dimension; note that in the process we also reduce the links of the simplices of smaller dimension but only by removing open stars of simplices with the same contractible links, so contractibility is preserved). Thus $\sys X^f_D$ is homotopy equivalent to $\sys X^f$. This completes the proof of the fact that $\sys X^f_D$ is contractible, hence systolic.

The action of the group of type preserving automorphisms of $X$ extends to $\sys X^f_D$. It remains to verify that the simplicial embedding $\psi\colon X_D\rightarrow \sys X^f_D$ is a quasi-isometry. To do that we construct the following quasi-inverse $\phi$ on the vertex set of $\sys X^f_D$. We first define $\phi$ on the set $E$ of these vertices of $\sys X^f_D$ that correspond to edges of $\sys X$. Let $e\in E$ be an edge of $\sys X$. If $e$ is an edge of $X$, then let $\phi(e)$ be the barycenter of $e$. Otherwise, if $e$ is an edge of friends, then let $\phi(e)$ be the barycenter of the edge $ab$ from Corollary~\ref{cor:friends}. Finally, if $e$ is an edge of acquaintances, then let $\phi(e)$ be the barycenter of the edge $ad,bc,$ or $bd$, from Corollary~\ref{cor:acquaint}. The map $\phi$ is Lipschitz. Indeed, it suffices to check that a pair of edges of triangle in $\sys X$ is sent to a pair of points at bounded distance. This follows from Lemmas~\ref{lem:cloopsI}(i,iii) and~\ref{lem:cloopsII}(i,iii),
together with Remark~\ref{rem:extra}. A vertex of $\sys X^f_D$ outside $E$ corresponds to a simplex $\sigma$ of $\sys X$ of dimension $\geq 2$. We define $\phi(\sigma)=\phi(e)$ for an arbitrary edge $e\subset \sigma$. It is easy to see that $\psi$ restricted to the vertex set of $X_D$ and $\phi$ are quasi-inverses.
\end{proof}

\section{Quotient construction}
\label{sec:TJan}
In this section we present a construction, suggested to us by Januszkiewicz, which might give rise to new systolic groups of cohomological dimension $3$. In order to illustrate the method, we first recall the following. Let $W$ be a Coxeter group of rank $4$ with all exponents finite and $\geq 3$. Then $W$ is the fundamental group of the simplex of groups $\mathcal{W}$ over the tetrahedron $T$, where all the local groups $W_I, \ I\subsetneq T$ are special Coxeter subgroups of $W$. If $I$ is a triangle, then $W_I=\Z_2$. If $I$ is an edge, then $W_I$ is a finite dihedral group. However, the vertex groups are infinite.

Let $X_i$ be the Coxeter complex of the vertex group $W_{i}$. The complex $X_i$ is systolic and for any normal finite index subgroup $W'_{ {i}}\subset W_{ i}$ avoiding the finite set (up to conjugacy) of elements with small translation length, the quotient $X_i'=W'_{ {i}}\backslash X_i$ is $6$--large and finite. Consider the simplex of groups $\mathcal{W}'$ over $T$ obtained from $\mathcal{W}$ by replacing each $W_{ {i}}$ with $W_{ {i}}'$. Then $\mathcal{W}'$ is \emph{locally $6$--large}, that is all vertex link developments are $6$--large, since they coincide with $X'_i$. By \cite[Thm~6.1]{JS}, the simplex of groups $\mathcal{W}'$ is developable. This means that $W_{ {i}}/W_{ {i}}'$ embed in the quotient $W'$ of $W$ by the normal closure of all the $W_{ {i}}'$. The group $W'$ acts geometrically on the development $T\times W'/\sim$, which is systolic. In particular $W'$ is a systolic group of cohomological dimension~$3$.

We would like to repeat this construction for a Coxeter group $W$ with an exponent $2$ as in Theorem~\ref{thm:dim3}. The group $W$ acts on the systolization $\sys X$ of the Coxeter complex $X$ from Construction~\ref{def:sys3}. The quotient complex $Y=W\backslash \sys X$ is equipped with the complex of groups structure $\mathcal Y$ coming from the stabilizers. The only infinite local groups are the vertex groups $W_{  i}$ at these vertices that come from the orbits of the original vertices of $X$. Each $W_{  i}$ acts on the systolic link $\sys X_i$ of the appropriate vertex in $\sys X$. Hence again for any normal finite index subgroup $W'_{ {i}}\subset W_{  i}$ avoiding a finite set of elements with small translation length, the quotient $\sys X_i'=W'_{ {i}}\backslash \sys X_i$ is $6$--large and finite. We form a complex of groups $\mathcal Y'$ over $Y$ replacing all $W_{  i}$ by $W'_{  i}$.

\begin{quest}
\label{quest:TJan}
Is the fundamental group of $\mathcal Y'$ systolic?
\end{quest}

Observe that the problem is that the action of $W$ on $\sys X$ has inversions, i.e.\ it stabilizes some simplices without fixing its vertices.
Hence the complex $Y$ has simplices that are not coming from the orbits of the original simplices of $\sys X$, but rather the simplices of its barycentric subdivision. Thus the combinatorial structure on the local developments differs from that of $\sys X'_i$. The property of being $6$--large is not inherited under subdivision, thus we cannot apply \cite[Thm~6.1]{JS}.

Note that if $\mathcal Y'$ were developable, then we could combine appropriate simplices of the subdivision to turn the development into a systolic complex and prove that $\pi_1 \mathcal Y'$ is systolic. Thus Question~\ref{quest:TJan} reduces to the question of developability of what we could call \emph{locally $6$--large orbi-complexes of groups}. The proof should not differ much from the proof of \cite[Thm~6.1]{JS}, but would require a reworking that is outside the scope of the current article.

\section{Appendix: CAT(1) metric on rank $3$ buildings}
\label{sec:app}

Let $W$ be an infinite triangle Coxeter group with finite exponents $(l,k,m)$ distinct from $(2,4,4), (2,4,5)$ and $(2,5,5)$. Assume $k\geq 3$ and either $l,m\geq 3$ or $l=2,m\geq 6$.
Let $X$ be the Coxeter realization of a building of type $W$. We equip each triangle of $X$ with the spherical $A_3$ metric, of angle $\frac{\pi}{2}$ at the vertex of type $\bf l$ and angle $\frac{\pi}{3}$ at the vertices of type $\bf k$ and $\bf m$. We consider the quotient pseudometric $d$ on $X$, see \cite[\S~I.5.19]{BH}. By \cite[Thm~I.7.19]{BH}, we have that $d$ is a complete geodesic metric on $X$. Moreover, since the vertex links of $X$ are graphs of girth $\geq 2\pi$, by \cite[Thm~II.5.2]{BH} the metric $d$ is locally CAT(1).

\begin{thm}
\label{thm:app:CAT(0)}
The metric $d$ on $X$ is $\mathrm{CAT(1)}$.
\end{thm}

In the proof we need the following result.

\begin{lemma}
\label{lem:length 6 loop}
Let $\gamma$ be a full cycle of length $6$ in a systolic complex $X$. Then there exists a vertex of $X$ adjacent to all the vertices of $\gamma$.
\end{lemma}
\begin{proof}
By \cite[Lem 4.2]{Els}, there is a simplicial map $F\colon \Delta \rightarrow X$, where $\Delta$ is a systolic disc and $F_{|\partial \Delta}=\gamma$.
By \cite[Lem 3.4(1)]{Els}, the disc $\Delta$ consists of at most $6$ triangles. On the other hand, since $\gamma$ is full, for any edge $e$ of the $6$ edges of $\partial \Delta$, the triangle of $\Delta$ containing $e$ intersects $\partial \Delta$ exactly in~$e$. Thus these $6$ triangles are distinct, cover whole $\Delta$, and every vertex of $\partial \Delta$ is contained in exactly two of them. Consequently, they have a common vertex $v\notin\partial \Delta$. Then $F(v)$ is adjacent to all the vertices of $\gamma$.
\end{proof}

\begin{proof}[Proof of Theorem~\ref{thm:app:CAT(0)}]
By \cite[p.\ 338--340]{Davis}, we can restrict to the case where $X$ is a single apartment.

By \cite[Thm~3.1.2]{Bow} it suffices to show that there is no closed path in $X$ of length $< 2\pi$ that is \emph{nonshrinkable}, i.e.\ not freely homotopic to the trivial path through paths of length $<2\pi$. For contradiction suppose that there exists a nonshrinkable path. By the cocompactness of the action of $W$ on $X$ and by Arzel\'a--Ascoli theorem, the infimum of the lengths of such paths is attained by a path $\gamma$. Such $\gamma$ is embedded by \cite[Thm~3.1.1]{Bow}. Moreover, $\gamma$ is a local geodesic.

We assume that the reader is acquainted with the method of
\cite[Ex~11.2]{EC} to classify local geodesics in piecewise
spherical complexes of our type. A~local geodesic of length
$<2\pi$ is contained in one of the listed in \cite[p.157]{EC}
annular galleries (see Figure~\ref{fig_annular}), M\"obius
galleries 
and necklace galleries build of the short edge $A$, the long
edge $B$, and beads $C,D,E$ (see Figure~\ref{fig_vertex}).
Since $X$ is oriented, M\"obius galleries are excluded. The
first listed annular gallery is excluded since $X$ is a
simplicial complex. The second listed annular gallery is
excluded since one of the boundary components of the annulus
forms a cycle of length $4$ of type $\bf{lklk}$: If $l\geq 3$,
then this contradicts the fact that $X$ is systolic, hence
$6$--large. If $l=2$, then this contradicts
Corollary~\ref{cor:convexstar}.

\begin{figure}[h]
 \begin{center}
{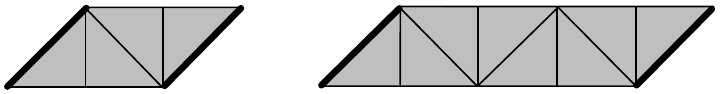}
\caption[rectangular]{Annular galleries. Heavy edges are identified.}
\label{fig_annular}
\end{center}
\end{figure}


\begin{figure}[h]
 \begin{center}
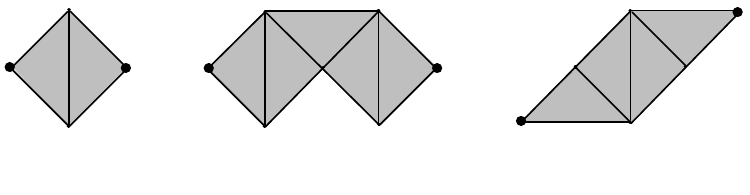
\caption[stargraph]{Beads.}
\label{fig_vertex}
\end{center}
\end{figure}

It remains to consider the case where $\gamma$ is contained in
a necklace gallery~$\mathcal{G}$. If $\mathcal{G}$ has a bead
$D$ or $E$, then we can homotope the subpath of $\gamma$ in
that bead into the boundary of the bead through paths of the
same length. Hence we can assume that $\mathcal G$ consists
only of edges $A$ and $B$ and the bead~$C$.

Consider first the case where $l\geq 3$. Then $X$ is systolic. Let $\Delta\subset X$ be the bounded complementary component of the Jordan curve $\gamma$ in the plane $X$. We consider $\Delta$ as a subcomplex of $X$ with cells subdivided along $\gamma$. Then $\Delta$ is systolic, in particular $6$--large.
The combinatorial length $|\partial \Delta|$ coincides with the number of maximal simplices of $\mathcal G$, which is $\leq 6$ by \cite[Ex~11.2]{EC}.
If $\partial \Delta$ is not full in $\Delta$, then three consecutive vertices of $\partial \Delta$ form a triangle. This contradicts the fact that $\gamma$ is a local geodesic. If $\partial \Delta$ is full in $\Delta$, then since $\Delta$ is $6$--large, we have $|\partial \Delta|=6$. By Lemma~\ref{lem:length 6 loop}, the path $\partial \Delta$ lies in the star of a vertex, which also contradicts the fact that $\gamma$ is a local geodesic.

We now treat the case where $l=2$. First suppose that $\mathcal
G$ consists only of edges $A$ and $B$, i.e.\ $\gamma$ is an
edge-path. By \cite[Ex~11.2]{EC} it has combinatorial length
$|\gamma|\leq 6$. Since $\gamma$ is a local geodesic, three
consecutive vertices of $\gamma$ cannot form a triangle. Thus
$\gamma$ cannot have a length $2$ subpath of type $\bf{k-2-m}$.
In particular $|\gamma|$ is even.

Let $\sys X$ be the systolization of $X$ from Construction~\ref{def:sys2}. If $|\gamma|=4$, then since $\sys X$ is systolic (Theorem~\ref{thm:dim2partCoxeter}), in particular $6$--large, there are two vertices of type $\bf k$ on $\gamma$ connected by an edge in $\sys X$. By Corollary~\ref{cor:convexstar}, entire $\gamma$ is contained in the star of a vertex of type $\bf 2$ in $X$, which contradicts the fact that it is a local geodesic.

If $\gamma$ is a cycle of length $6$, then it is of type $\bf {2k2k2k}$ or $\bf{2m2m2m}$. The first type is excluded by Lemma~\ref{le:2kloops}. If $\gamma$ is of the second type, then since $\sys X$ is systolic, Lemma~\ref{lem:length 6 loop} implies that $\gamma$ is contained in the star of a vertex of type $\bf k$, contradicting the fact that it is a local geodesic.

Finally suppose that $\mathcal G$ contains a bead $C$. To be locally geodesic at the vertex where it leaves $C$, the path $\gamma$ must enter another bead $C$. Consequently, $\mathcal G$ is formed entirely of beads $C$. Then $\gamma$ is also a local geodesic in the piecewise Euclidean or hyperbolic metric for which $X$ is isometric with $\R^2$ or $\Hyp^2$, which is a contradiction.
\end{proof}

Having established that $X$ is CAT(1), we now prove a criterion
for $\pi$--convexity of its subcomplex. A subcomplex $Y\subset
X$ is \emph{$\pi$--convex} if for any points $x,x'\in Y$ at
distance $d(x,x')<\pi$, the unique geodesic $xx'$ is contained
in~$Y$.

We need the following consequence of \cite[Ex~11.2]{EC}. We
discuss galleries outside $Y$ joining $Y$ to itself. Such a
gallery might start (and end) with a vertex of $Y$ or an edge
of $Y$. In particular it might start (and end) with what we
will call \emph{the bead $\sqrt{C}$}, which is a single
triangle whose long edge $B$ is contained in $Y$.

\begin{cor}
\label{cor:galleries}
Let $X$ be a complex whose vertices are of two types, $\bf l,k$, and which is build of triangles of type $\bf{k-l-k}$ with the spherical $A_3$ metric, of angle $\frac{\pi}{2}$ at the vertex of type $\bf l$ and angle $\frac{\pi}{3}$ at the vertices of type $\bf k$. Let $Y\subset X$ be a subcomplex. Let $\gamma$ be a path disjoint from $Y$ except at its endpoints, which cannot be perturbed to a path of smaller length with that property.
If $\gamma$ has length~$<\pi$, then it is contained in one of \emph{rectangular} galleries listed in Figure~\ref{fig_3} and necklace galleries
$B, B^2, BA, BA\sqrt C, A,A^2, A^3, A^2\sqrt C, A\sqrt C,$ $AC, \sqrt{C},$ $\sqrt{C}\sqrt{C}, \sqrt C C,C$ (see Figure~\ref{fig_4}).
\end{cor}

\begin{figure}[h]
 \begin{center}
{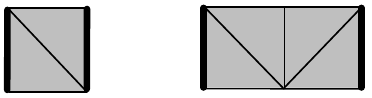}
\caption[rectangular]{Rectangular galleries. Heavy edges are contained in $Y$.}
\label{fig_3}
\end{center}
\end{figure}

\begin{figure}[h]
 \begin{center}
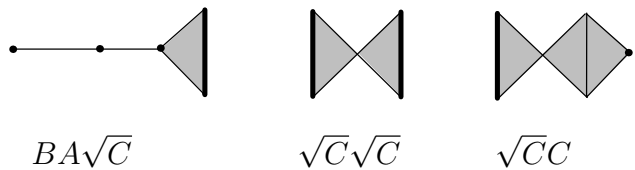
\caption[stargraph]{Some of the necklace galleries in Corollary~\ref{cor:galleries}.}
\label{fig_4}
\end{center}
\end{figure}

\begin{proof}
First observe that if $\gamma$ starts (or ends) in the interior
of an edge $e$ of~$Y$, then by the length minimality assumption
$\gamma$ hits $e$ at angle $\frac{\pi}{2}$. Let $\sigma, \tau$
be the simplices of $X$ in the interior of which $\gamma$
starts and ends. We consider the \emph{double} of $X$ along
$\sigma\cup \tau$, which is the following complex. Let $X'$ be
a copy of $X$ and let $\sigma',\tau', \gamma'$ be the copies of
$\sigma, \tau, \gamma$ in $X'$. The \emph{double} $X\cup X'$ is
the amalgam of $X$ and $X'$ along $\sigma\cup \tau$ identified
with $\sigma'\cup \tau'$. Then the closed path $\gamma\cup
\gamma'$ is a local geodesic in $X\cup X'$ of length $<2\pi$.
Thus $\gamma\cup \gamma'$ is contained in one of the galleries
listed in \cite[p.157]{EC}. The fact that such a gallery in
$X\cup X'$ is the double of a gallery in $X$ restricts its
form. M\"{o}bius galleries are excluded, and annular galleries
in Figure~\ref{fig_annular} are doubles of rectangular
galleries in Figure~\ref{fig_3}. The beads~$D,E$ are not
doubles in the sense that they do not admit a symmetry
interchanging their endpoints and fixing an edge pointwise.
Consequently, necklace galleries are doubles of exactly these
galleries that we have listed in the statement of the
corollary.
\end{proof}

A \emph{fan} of triangles at a vertex $y\in X$ is a subcomplex of $\st_X(y)$ that is a join of $y$ with a path graph in $\lk_X(y)$. The \emph{length} of the fan is the length of the path graph. An \emph{outer fan} of a subcomplex $Y\subset X$ at a vertex $y\in Y$ is a fan whose path graph is disjoint from $\lk_Y(y)\subset \lk_X(y)$ except at its endpoints.

We assume again that $X$ is the Coxeter realization of a building, with the metric $d$ defined in the first paragraph of the appendix.

\begin{defin}
A simplicial map $\varphi\colon Y\rightarrow X$ is a \emph{local embedding} at a vertex $y\in Y$ if $\varphi_{|\st_Y(y)}$ is an embedding.
The map $\varphi$ is \emph{strongly locally convex} at a vertex $y\in Y$, if it is a local embedding at $y$ and all of the following conditions hold.
We pull back the vertex types and angles from $X$ to $Y$ via~$\varphi$.
\begin{itemize}
\item
The map $\varphi$ is \emph{locally convex} at $y$, i.e.\ if the angle at $y$ is $\frac{\pi}{n}$, then each outer fan of triangles of $\varphi\big(\st_Y(y)\big)\subset \st_X(\varphi(y))$ at $\varphi(y)$ has length~$\geq n$.
\item
If $l\geq 3$ and the type of $y$ is $\bf l$, then each outer fan of triangles of $\varphi\big(\st_Y(y)\big)\subset \st_X(\varphi(y))$ at $\varphi(y)$ has length $\geq 3$.
\item
If $l=2$ and the type of $y$ is $\bf m$, then each outer fan of triangles of $\varphi\big(\st_Y(y)\big)\subset \st_X(\varphi(y))$ at $\varphi(y)$ has length $\geq 7$.
\end{itemize}
The map $\varphi$ is a \emph{local embedding} (respectively, \emph{locally convex, strongly locally convex}) if it is a local embedding (respectively, locally convex, strongly locally convex) at every vertex $y\in Y$.
\end{defin}

\begin{thm}
\label{thm:app:convex}
Let $Y$ be a connected simplicial complex. Suppose that $\varphi\colon Y\rightarrow X$ is strongly locally convex.
Then $\varphi$ is an embedding and $\varphi(Y)\subset X$ is $\pi$--convex.
\end{thm}

In the proof we will translate strong local convexity to the following combinatorial convexity from \cite{JS}.
A subcomplex $Y$ of a flag simplicial complex~$X$ is \emph{$3$--convex} if it is full and every vertex $x\in X$ adjacent to nonadjacent vertices $y,y'\in Y$ lies in $Y$. In other words, there is no geodesic edge-path of length $1$ or $2$ disjoint from $Y$ except at its endpoints.
A local embedding $\varphi\colon Y\rightarrow X$ is \emph{locally $3$--convex} if $\varphi\big(\lk_Y(y)\big)\subset \lk_X(\varphi(y))$ is $3$--convex for every vertex $y\in Y$.

\begin{proof}[Proof of Theorem~\ref{thm:app:convex}]
Consider first the case where $l\geq 3$. Then $X$ is systolic. Since $\varphi$ is strongly locally convex and $l\geq 3$, the outer fans of $\varphi\big(\st_Y(y)\big)\subset \st_X(\varphi(y))$ at $\varphi(y)$ for any vertex $y\in Y$ have length $\geq 3$. Hence $\varphi$ is locally $3$--convex. By \cite[Prop~4.3]{JS}, the map $\varphi$ is an embedding and we can and will identify $Y$ with $\varphi(Y)$. Moreover, $Y\subset X$ is $3$--convex by \cite[Lem~7.2(3)]{JS}.

If $Y\subset X$ is not $\pi$--convex, then there is a geodesic of length~$<\pi$ disjoint from $Y$ except at its endpoints.
The infimum of lengths of such geodesics is bounded away from zero by the local convexity of $Y$.
By Arzel\'a--Ascoli theorem, this infimum is attained by a geodesic $\gamma$. Then $\gamma$ satisfies the hypothesis of Corollary~\ref{cor:galleries}.
In this application we identify the two types $\bf m,k$ to one type $\bf k$. Thus $\gamma$ is contained in a gallery $\mathcal{G}$ that is of the form listed in Corollary~\ref{cor:galleries}. Denote by $\sigma,\tau$ the simplices of $Y$ in the interior of which lie the endpoints of~$\gamma$.

Since $Y\subset X$ is $3$--convex, $\mathcal{G}$ is neither rectangular nor consists of $\leq 2$ maximal simplices. Thus $\mathcal{G}$ has the form $BA\sqrt C,A^3,A^2\sqrt C, AC$ or $\sqrt C C$. In the first three cases, let $e$ be the edge of $\mathcal{G}$ disjoint from $Y$, which is the short edge $A$. By the Projection Lemma \cite[Lem~7.7]{JS} applied to $e$, there is a vertex $y\in Y$ forming a triangle with $e$. By \cite[Lem~7.7]{JS} applied to the endpoints of $e$, the vertex $y$ is also adjacent to $\sigma$ and $\tau$. In particular, by the $2$--dimensionality of $X$, we have that $\sigma$ and $\tau$ are vertices. Thus at the vertex $x$ of $e$ of the type distinct from $\bf l$ the path $\gamma$ passes between two triangles sharing the edge $xy$. This contradicts the fact that $\gamma$ is a local geodesic at~$x$.

If $\mathcal{G}$ has the form $AC$ or $\sqrt C C$, then the middle triangle $\delta$ of $\mathcal{G}$ intersects $Y$: otherwise applying \cite[Lem~7.7]{JS} to $\delta$ would contradict the $2$--dimensionality of $X$. Thus a vertex $y$ of $\delta$ belongs to $Y$. Let $x\in \delta$ be the vertex of type $\bf l$. By \cite[Lem~7.7]{JS} applied to $x$, the vertex $y$ is adjacent to~$\sigma$. In particular $\sigma$ is a vertex, so that $\mathcal{G}$ has the form $AC$. This contradicts the fact that $\gamma$ is a local geodesic at $x$.

We now treat the case where $l=2$. Consider the systolization $X\subset \sys X$ from Construction~\ref{def:sys2}. We extend $Y$ to a complex $\sys Y$ mapping to $\sys X$ in the following way.

First, for each vertex $w\in Y$ of type $\bf 2$, let $N(w)\subset Y$ be the join of $w$ with the discrete set of its neighbors of type $\bf m$ in $Y$. We attach to $Y$ along $N(w)$ the join of $N(w)$ and a copy of the discrete set of neighbors of $\varphi(w)$ of type $\bf k$ that are outside $\varphi\big(\st_Y(w)\big)$. We denote this extension of $Y$ by $Y'$. Note that $\varphi$ extends in an obvious way to a local embedding $\varphi\colon Y'\rightarrow X$. For each vertex $y\in Y'$, since $\varphi_{|Y}$ was strongly locally convex, we have that $\varphi$ is
locally convex at $y$. Moreover, if $y$ is of type~$\bf m$, then each outer fan of triangles of $\varphi\big(\st_{Y'}(y)\big)\subset \st_X(\varphi(y))$ at $\varphi(y)$ has length $\geq 6$. This follows from the fact that this fan is contained in an outer fan of $\varphi\big(\st_{Y}(y)\big)$, and its path graph is obtained from the path graph of the latter fan by removing the extreme edges if they end with a vertex of type $\bf 2$.

Next, consider again each vertex $w\in Y$ of type $\bf 2$ and $N(w)\subset Y$. Let $\mathcal{V}(w)\subset Y'$ denote the set of neighbors of $w$ of type~$\bf k$. Attach to $Y'$ along the join of $N(w)$ with $\mathcal{V}(w)$ the join of $N(w)$ with the simplex spanned on $\mathcal{V}(w)$. We denote this extension of $Y'$ by $\sys Y$.
We keep the notation $\varphi$ for the obvious extension $\varphi\colon \sys Y\rightarrow \sys X$.

\begin{claim} The map $\varphi\colon \sys Y\rightarrow \sys X$ is locally $3$--convex.
\end{claim}

\begin{proof}
We begin with showing that $\varphi$ is a local embedding.
Since $\varphi_{|Y'}$ is a local embedding, for a vertex $v\in Y'$ of type $\bf k$ we only need to exclude the possibility that it is connected by distinct edges in $\sys Y$ to $v',v''$ of type~$\bf k$, satisfying $\varphi(v')=\varphi(v'')$. Let $w',w''\in Y$ be the vertices of type~$\bf 2$ with $v,v'\in \mathcal{V}(w')$ and $v,v''\in \mathcal{V}(w'')$. Note that $w'\neq w''$ since $\varphi_{|Y'}$ is a local embedding at $w'$. However, by Corollary~\ref{cor:convexstar} applied to $\varphi(v)$ and $\varphi(v')=\varphi(v'')$, we have $\varphi(w')=\varphi(w'')$. This contradicts the fact that $\varphi_{|Y'}$ is a local embedding at~$v$. Thus $\varphi$ is a local embedding at vertices of type $\bf k$. In particular there are no double edges of type $\bf k-k$ in $\sys Y$ and consequently $\sys Y$ is a simplicial complex. At a vertex of type $\bf 2$ or $\bf m$ the fact that $\varphi$ is a local embedding follows immediately from the fact that $\varphi_{|Y'}$ is a local embedding and the fact that $\sys Y$ is a
simplicial complex.

We now prove that $\varphi$ is locally $3$--convex. We first consider a vertex $w\in Y$ of type $\bf 2$ and the link $\lk_{\sys{X}}(\varphi(w))=\sys X_2$ from the proof of Lemma~\ref{lem:vertexlinks}. Recall that $\sys X_2$ is a join of a simplex of vertices of type~$\bf k$ with a discrete set of vertices of type~$\bf m$. By the construction of $Y'$ and $\sys{Y}$, the entire simplex of vertices of type~$\bf k$ is contained in $\varphi\big(\lk_{\sys{Y}}(w)\big)$. Thus $\varphi\big(\lk_{\sys{Y}}(w)\big)\subset \lk_{\sys{X}}(\varphi(w))$ is $3$--convex.

Next, consider a vertex $v\in Y'$ of type $\bf k$. Since the angle at~$v$ is $\frac{\pi}{3}$, and $\varphi_{|Y'}$ is locally convex at $v$, we have that $\varphi\big(\lk_{Y'}(v)\big)\subset \lk_X(\varphi(v))=X_k$ is $3$--convex.
By the construction of $Y'$ and $\sys Y$, the subcomplex $\varphi\big(\lk_{\sys Y}(v)\big)\subset \lk_{\sys X}(\varphi(v))=\sys X_k$ coincides with the preimage of $\varphi\big(\lk_{Y'}(v)\big)$ under the map $f\colon \sys X_k\rightarrow X_k$ from the proof of Lemma~\ref{lem:vertexlinks}. Since $f$ satisfies the hypothesis of Lemma~\ref{le:collaps}, it is easy to see that the preimage under $f$ of a $3$--convex subcomplex is $3$--convex as well.

Finally, consider a vertex $u\in Y$ of type $\bf m$. As in the proof of Lemma~\ref{lem:vertexlinks}, the link $\sys X_m$ consists of a subcomplex $V$ of vertices of type $\bf k$ and cones over some simplices of $V$ with cone-point of type $\bf 2$. By the construction of $Y'$ and $\sys Y$, if $\varphi\big(\lk_{\sys Y}(u)\big)\subset \lk_{\sys X}(\varphi(u))=\sys{X}_m$ contains such a cone-point, then it contains the entire cone. Thus to show that $\varphi\big(\lk_{\sys Y}(u)\big)\subset \sys{X}_m$ is $3$--convex, it suffices to show that $\varphi\big(\lk_{\sys Y}(u)\big)\cap V\subset V$ is $3$--convex. A geodesic path of length $1$ or $2$ in $V$ that is disjoint from $\varphi\big(\lk_{\sys Y}(u)\big)$ except at its endpoints would correspond to an outer fan of $2$ or $4$ triangles of $\varphi\big(\lk_{Y'}(u)\big)\subset X_m$ at $\varphi(u)$. But we have excluded such fans using strong local convexity in the discussion after the definition of $Y'$.
\end{proof}

By the claim and \cite[Prop 4.3]{JS} it follows that $\varphi$ is an embedding. Thus by \cite[Lem~7.2(3)]{JS}, the image $\varphi(\sys Y)$ is $3$--convex in $\sys{X}$. From now on we identify $\sys Y$ with $\varphi(\sys Y)\subset \sys X$. We define $\gamma$ as in the case $l\geq 3$ and apply Corollary~\ref{cor:galleries} to restrict the form of the gallery $\mathcal{G}$ containing $\gamma$.

If $\mathcal{G}$ is rectangular or consists of $\leq 2$ maximal simplices, then it must lie in $\sys{Y}$ by the $3$--convexity of $\sys{Y}$. Then $\gamma$ lies in one of the cones over $N(w)$ attached to $Y$ to form $Y'$. But such a cone does not contain beads $C$ and $\sqrt{C}$. Moreover, in such a cone the only edge-path disjoint from $Y$ except at its endpoints passes through the cone point, which is of type $\bf k$ and where consequently this edge-path fails to be a local geodesic. Thus $\mathcal{G}$ consists of $3$ maximal simplices.

Since $l=2$, the bead $A$ cannot follow the bead $C$ or $\sqrt C$. Thus $\mathcal{G}$ has the form $A^3$.
Since $\gamma$ does not contain a length $2$ edge-path of type $\bf{k-l-m}$, it is of one of the types $\bf {2m2m, 2k2k}$.
If $\gamma$ is of type $\bf {2m2m}$, then $\gamma$ is disjoint from $\sys Y$ except at its endpoints and we argue using the Projection Lemma as in the case $l\geq 3$. If $\gamma$ is of type $\bf {2k2k}$, then denote by $w,v$ the endpoints of $\gamma$ of types $\bf 2,k$, respectively.
By the definition of $Y'$, the middle vertex $x$ of $\gamma$ of type $\bf k$ lies in $Y'$ and its unique neighbor of type $\bf 2$ in $Y$ is $w$. Since $\sys{Y}\subset \sys {X}$ is full, the edge $xv$ also lies in $\sys{Y}$. By the construction of $\sys{Y}$, the vertices $x$ and $v$ have a common neighbor of type $\bf 2$ in $Y'$, which thus needs to be $w$. Hence $\gamma$ backtracks, which is a contradiction.
\end{proof}

\end{document}